\documentclass[11pt,a4paper]{article}

\usepackage[T1]{fontenc}
\usepackage[utf8]{inputenc}
\usepackage{amsmath,amssymb,amsthm, latexsym,color,epsfig,enumerate,a4, graphicx}
\usepackage{xspace}
\usepackage{tikz}
\usetikzlibrary{patterns}
\usepackage{comment}
\usepackage{enumitem}
\usepackage{lmodern}
\usepackage{microtype}
\usepackage{a4wide}
\usepackage{hyperref}

\newcommand{\cupdot}{\mathbin{\mathaccent\cdot\cup}}

\usepackage{appendix}

\theoremstyle{plain}
\newtheorem{theorem}{Theorem}[section]
\newtheorem{lemma}[theorem]{Lemma}
\newtheorem{claim}[theorem]{Claim}

\theoremstyle{definition}
\newtheorem{definition}[theorem]{Definition}

\newcommand{\Bin}{\ensuremath{\textrm{Bin}}}

\newcommand{\eps}{\ensuremath{\varepsilon}}

\newcommand{\calL}{\ensuremath{\mathcal L}}
\newcommand{\calE}{\ensuremath{\mathcal E}}
\newcommand{\calI}{\ensuremath{\mathcal I}}

\newcommand{\calG}{\ensuremath{\mathcal G}}

\newcommand{\Exp}{\ensuremath{\mathbb E}}
\newcommand{\Gnp}{\ensuremath{\mathcal G(n,p)}}

\newcommand{\CN}{\ensuremath{N^*}}

\newcommand{\hX}{\hat{X}}
\newcommand{\hS}{\hat{S}}
\newcommand{\teps}{\tilde{\eps}}
\newcommand{\calH}{\mathcal{H}}
\newcommand{\calJ}{\mathcal{J}}
\newcommand{\calT}{\mathcal{T}}
\newcommand{\calB}{\mathcal{B}}

\title{Monochromatic cycle covers in random graphs}
\author{
D\'aniel Kor\'andi \thanks{Institute of Mathematics,
EPFL, Lausanne, Switzerland. Email: {\tt
daniel.korandi@epfl.ch.} Research supported in part by SNSF grants 200020-162884 and 200021-175977.}
\and
Frank Mousset \thanks{Department of Computer Science,
ETH, 8092 Z\"urich, Switzerland. \newline Email: {\tt \{frank.mousset|nskoric\}@inf.ethz.ch.}} 
\and
Rajko Nenadov 
\thanks{School of Mathematical Sciences,
Monash University, VIC 3800, Australia.  \newline Email: {\tt rajko.nenadov@monash.edu.}}
\and
Nemanja \v Skori\'c $^\dag$
\and
Benny Sudakov \thanks{Department of Mathematics,
ETH, 8092 Z\"urich, Switzerland. Email: {\tt benjamin.sudakov@math.ethz.ch.}
Research supported in part by SNSF grant 200021-175573.}
}

\begin{document}
\maketitle

\abstract{
A classic result of Erd\H{o}s, Gy\'arf\'as and Pyber states that
for every coloring of the edges of $K_n$ with $r$ colors, there is a cover
of its vertex set by at most $f(r) = O(r^2 \log r)$ vertex-disjoint monochromatic
cycles. In particular, the minimum number of such covering cycles does not depend on the size of $K_n$ but only on the number of colors.
We initiate the study of this phenomenon in the case where $K_n$ is replaced by
the random graph $\Gnp$. Given a fixed integer $r$ and $p =p(n) \ge n^{-1/r + \eps}$,
we show that with high probability the random graph $G \sim \Gnp$ has the
property that for every $r$-coloring of the edges of $G$, there is a collection of
$f'(r) = O(r^8 \log r)$ monochromatic cycles covering all the vertices of $G$. Our
bound on $p$ is close to optimal in the following sense: if $p\ll (\log n/n)^{1/r}$, then
with high probability there are colorings of $G\sim\Gnp$ such that the number 
of monochromatic cycles needed to cover all vertices of $G$ grows with $n$.}

\section{Introduction}

In this paper,
we consider a question of the following type: For a certain class of graphs
$\calG$, is it true that the vertex set of every $r$-edge-colored graph $G\in
\calG$ can be covered with a number of monochromatic paths or cycles\footnote{Single edges and vertices count as ``degenerate'' cycles for the purposes of this paper.} that only
depends on the number of colors $r$?

The study of such questions dates back to the 1960s, when Gerencs\'er and
Gy\'arf\'as \cite{gyarfas67} observed that every $2$-coloring of the edges of
the complete graph $K_n$ contains two vertex-disjoint monochromatic paths that
together cover all vertices of the graph.
Later Gy\'arf\'as \cite{gyarfas1989covering} conjectured that the analogous
statement for $r$ colors should also be true, namely, that every $r$-edge-colored
$K_n$ can be covered with $r$ vertex-disjoint monochromatic paths. He made a
step towards this conjecture by showing that there is always a cover that
uses $O(r^4)$ (not necessarily disjoint) monochromatic paths. The case $r = 3$
was only recently resolved by Pokrovskiy \cite{pokrovskiy2014partitioning} and
for $r \ge 4$ the conjecture remains open. 

Strengthening Gy\'arf\'as' result,  Erd\H{o}s, Gy\'arf\'as and Pyber
\cite{erdos1991vertex} showed that the vertices of every $r$-colored $K_n$ can
be covered by $O(r^2 \log r)$ vertex-disjoint monochromatic \emph{cycles}. It
is worth noting that their proof is one of the first applications of the
\emph{absorbing method}, a technique that has turned out to be extremely useful
for embedding-type problems.
This was subsequently improved to $O(r \log r)$
cycles by Gy\'arf\'as, Ruszink\'o, S\'ark\"ozy and Szemer\'edi
\cite{gyarfas2006improved}. For $r=2$, Lehel conjectured that, just like
for paths, the vertices can be covered by two vertex-disjoint monochromatic
cycles of different colors. This was eventually established by Bessy and
Thomass\'e \cite{bessy2010partitioning}.
Some generalizations of these results
concerning more complicated graphs other than paths or cycles were obtained in
\cite{grinshpun2016monochromatic,sarkozy2013improved}.
Similar properties of host graphs other than complete
 graphs were also studied: complete bipartite graphs are considered in
\cite{gyarfas1989covering,haxell95,lang2017local}, complete graphs with
only few edges missing in \cite{gyarfas97nearlycompl}, graphs
with large minimum degree in
\cite{balogh2014partitioning,debiasio2017monochromatic,letzter2015monochromatic}
and graphs with small independence number in \cite{sarkozy2011monochromatic}.
For further results and
research directions we refer the reader to the recent survey by Gy\'arf\'as
\cite{gyarfas2016vertex}.

In this paper, we consider the same problem in the setting of the binomial
random graph model $\Gnp$. The study of covering $\Gnp$ by monochromatic
pieces was initiated by Bal and DeBiasio \cite{bal2015partitioning}, who showed
that if $p = \tilde \Omega(n^{-1/3})$ then with high probability (w.h.p), $G
\sim \Gnp$ has the property that every $2$-coloring of the edges of $G$
contains two vertex-disjoint monochromatic trees that cover its vertex set.
They proposed a conjecture that already $p \gg (\log n / n)^{1/2}$ suffices,
which was recently verified by Kohayakawa, Mota and Schacht
\cite{kohayakawa2016monochromatic}.
Here we
continue this line of research by studying random analogs of the theorems of
Gy\'arf\'as \cite{gyarfas1989covering} and Erd\H{o}s, Gy\'arf\'as and Pyber \cite{erdos1991vertex}. Our main result reads as follows:

\begin{theorem} \label{thm:mainr2} 
  Fix $\eps > 0$ and an integer $r \ge
  2$. If $p=p(n) > n^{-1/r + \eps}$, then $G \sim \Gnp$  w.h.p has the property that
  for every $r$-coloring of the edges of $G$, there is a collection of at most
  $(100r)^8 \log r$ monochromatic cycles covering all the vertices of $G$.
\end{theorem}

Although we believe that it should be possible to choose the cycles so that
they are vertex-disjoint, our result does not give this. We remark that the
bound on $p$ in the theorem is almost best possible. Indeed, a result of
Bal and DeBiasio \cite[Theorem 1.7]{bal2015partitioning} shows that for $p \ll
(\log n / n)^{1/r}$ w.h.p there exists an $r$-coloring of $G\sim\Gnp$ which
requires an unbounded number of monochromatic components (and in particular,
cycles) to cover all vertices. Their coloring is based on the fact that such
a $G$ contains an independent set $X$ of unbounded size with the property
that every vertex has at most $r-1$ neighbors in $X$. Now one can color
all edges outside of $X$ with the color $r$, and for every $v\in V(G)\setminus X$ 
color the edges from $v$ to $X$ using each of the colors from $[r-1]$
at most once. It is not difficult to verify that every monochromatic component
can cover at most one vertex of $X$, and so at least $|X|$ such components are
needed to cover the whole graph.
With this in mind, it seems likely that $(\log n/n)^{1/r}$ is the correct order of magnitude of
the threshold for the property of always having a cover by a bounded number of
monochromatic cycles.

We remark that Theorem \ref{thm:mainr2} is an example of the more general
phenomenon that sufficiently dense -- but still very sparse -- random graphs
$\Gnp$ often have global properties that are remarkably similar to those of the
(much denser) complete graph $K_n$. Transferring classic results about complete
graphs to the random graph setting is an active line of research with some of
the milestones achieved only recently (see the survey by Conlon
\cite{conlon2014combinatorial}).

\paragraph{Structure of the paper}
The paper is organized as follows. In the next section we give the
proof of Theorem \ref{thm:mainr2}, assuming two key lemmas: the first one
shows that we can cover all but $O(1/p)$ vertices, while the other one takes
care of the remaining
vertices. In Section \ref{sec:tools} we collect some tools and properties of
random graphs that are used in the proof of these lemmas. The two lemmas
are then proved in Sections \ref{sec:almost_covering} and
\ref{sec:covering_small_set},
respectively. In the last section we discuss some open problems and future
research directions.

\paragraph{Notation}
We use the common notation $[r] = \{1, \dots, r\}$ for the first $r$ positive
integers. Instead of saying that a set has size $r$, we sometimes say that it
is a \emph{$r$-set}. We occasionally write $A=B_1\cupdot \dots\cupdot B_t$ to mean
that $B_1,\dots,B_t$ form a partition of $A$.
For $a,b>0$, we write $a\pm b$ to denote the interval
$[a-b,a+b]$. 

If $G$ is a graph, we let $N_G(v)$ denote the neighborhood of $v$, that is,
$N_{G}(v) = \{ w : \{v,w\} \in E(G) \}$. Similarly, if $A\subseteq V(G)$, then
$N_G(A)=\bigcup_{v\in A}N_G(v)$ is the neighborhood of $A$. On the other hand,
we denote by $\CN_G(A) = \bigcap_{v\in A}N_G(v)$ the \emph{common} neighborhood of
$A$.
For subsets $A,B \subseteq V(G)$ and a vertex $v\in V(G)$, we write $N_G(v,B)$,
$N_G(A,B)$ and $\CN_G(A,B)$ for the sets $N_G(v)\cap B$, $N_G(A)\cap B$ and
$\CN_G(A)\cap B$, respectively.
If $A,B\subseteq V(G)$ are disjoint, we write $e_G(A,B)$ for number of edges
with one endpoint in $A$ and another in $B$. If it is clear which graph $G$ we are talking about, we omit the
subscript $G$ in the above notations. 

Given a coloring of the edges of $G$ with
colors $\{1, \dots, r\}$, we define $G_i$ to be the spanning subgraph of $G$
containing the edges of color $i$. In this setting, we abbreviate $N_{G_i}(v)$,
$N_{G_i}(A,B)$, $\CN_{G_i}(A)$, etc.\ by $N_i(v)$, $N_i(A, B)$,
$\CN_i(A)$, and so on. So for example $N_i(v,B)$ is the neighborhood of $v$ in
the set $B$ via edges of color $i$.

The vertex set of the random graph $\Gnp$ is understood to be $[n]$. We say that $G\sim \Gnp$ satisfies some property with high probability (w.h.p) if the property holds for $G$ with probability tending to 1 as $n$ tends to infinity. 

We routinely omit floor and ceiling signs if they are not
essential.

\section{Proof of Theorem \ref{thm:mainr2}}

Theorem \ref{thm:mainr2} states that if the edges of $G\sim\Gnp$ (where $p> n^{-1/r+\eps})$ are colored with $r$ colors, then there are $O(r^8\log r)$ monochromatic cycles covering all its vertices. We use the following strategy to find such cycles. 

Our first step is to partition the vertex set of $G$ into
$s=(101r)^4$ disjoint sets whose sizes differ by at most 1:
\[V(G) = W_1\cupdot\dots\cupdot W_s,\]
where each $W_i$ has size at most $n/(100r)^4$. 
Next, we consider one particular set $W_i$ and try to find $O(r^4\log r)$
monochromatic cycles in $G$ that cover all the vertices in $W_i$. Importantly
for our proof method, these cycles can (and will) use vertices \emph{outside} of $W_i$. Since
there are $s=O(r^4)$ sets, finding such cycles for each $W_i$ independently results in a cover of all vertices by
$O(r^8\log r)$ cycles (although many vertices might be covered multiple times).

We cover the vertices in the set $W_i$ in two steps: first, we cover all but $O(1/p)$ vertices, and then we cover the remaining ones. 
Here the quantity $1/p$ comes into play as the ``threshold'' size of a set $X$ to expand to all other vertices (note that each vertex has roughly $np$ neighbors). Indeed, the proof of the first step relies on the fact that every vertex set of size $\Omega(1/p)$ is adjacent to $\Omega(n)$ other vertices. On the other hand, it is key to our second step that the individual neighborhoods of $O(1/p)$ vertices are almost disjoint.
In any case, our arguments for the two steps are entirely different, so it is natural to split the proof accordingly.
More precisely, we establish the following two lemmas.

\begin{lemma}
  \label{lemma:smallset-cover}
  For every integer $r \ge 2$, there is a constant $K=K(r)>0$ such that the following
  holds. Let $W \subseteq [n]$ be a fixed set of at most $n / (100 r)^4$ vertices,
  and let $G\sim \Gnp$ where $p=p(n) \gg (\log n / n)^{1/2}$. Then w.h.p for
  every $r$-coloring of the edges of $G$, there is a collection of $3r^2$
  vertex-disjoint monochromatic cycles that cover all but at most $K/p$
  vertices of $W$.
\end{lemma}

\begin{lemma}
  \label{lemma:tinyset-cover}
  Let $r \ge 2$ be a fixed integer and let $\eps,K>0$ be some constants. Let $W \subseteq [n]$ be a fixed set of at most $n/2$ vertices, and let $G \sim \Gnp$ where $p =p(n) \ge n^{-1/r+\eps}$. Then w.h.p in every $r$-coloring of the edges of $G$, every subset $Q \subseteq W$ of size at most $K / p$ can be covered by $400 r^4 \log (4r^2)$ monochromatic cycles.
\end{lemma}

Observe that while Lemma \ref{lemma:smallset-cover} provides an approximate
cover of one fixed set $W$, Lemma \ref{lemma:tinyset-cover} applies to all
small subsets $Q \subseteq W$ simultaneously.
We note that Lemma \ref{lemma:smallset-cover} applies to a somewhat larger probability range 
than we actually need it to be; also, the lemma gives vertex-disjoint cycles, but we do not really use this fact.

Combining the two lemmas, we immediately get that w.h.p each fixed set $W_i$
can be covered by $3r^2+400r^2\log(4r^2)$ monochromatic cycles. By the union
bound, this is true for each of the constantly many sets $W_1,\dots,W_s$, and so
we can cover all of $V(G)$ using
\[ s( 3r^2+ 400 r^4 \log (4r^2)) \le (101r)^4\cdot 1000r^4\log r \le
  (100r)^8\log r \]
monochromatic cycles. This completes the proof of Theorem
\ref{thm:mainr2}, although of course we still have to prove
Lemmas \ref{lemma:smallset-cover} and \ref{lemma:tinyset-cover}.

\section{Tools and preliminaries} \label{sec:tools}

In the proof of Lemma \ref{lemma:tinyset-cover} we use the following
generalization of the result of Erd\H{o}s, Gy\'arf\'as and Pyber \cite{erdos1991vertex} to graphs
with a given independence number:

\begin{theorem}[S\'ark\"ozy {\cite{sarkozy2011monochromatic}}] \label{thm:indep}
  If the edges of a graph $G$ with independence number $\alpha$ are colored with
  $r$ colors then $G$ contains a collection of at most $25 (\alpha r)^2 \log(\alpha r)$
  vertex-disjoint monochromatic cycles covering the vertex set of $G$.
\end{theorem}

Of course, we cannot apply this directly to $G\sim \Gnp$ because $\alpha(G)$ is
w.h.p unbounded (unless $p$ is very close to 1). Instead, we will use it to
find cycle covers in a certain auxiliary graph. To turn the cycles from the
auxiliary graph into real cycles in $G$, we use a generalization of Hall's
criterion, due to Haxell \cite{haxell95}, for the existence of saturating matchings in hypergraphs (see Section \ref{sec:covering_small_set}). Given a family $\calE$ of subsets of some ground set $V$, the \emph{vertex cover
number} $\tau(\calE)$ is the smallest size of a set $X \subseteq V$
intersecting every set in $\calE$.

\begin{theorem}[Haxell {\cite{haxell95}}] \label{thm:hyp_matching}
Let $\{ H_i = (V, \calE_i) \}_{i \in \calI}$ be a family of $r$-uniform hypergraphs on the same vertex set, for some positive integer $r$. If $\tau(\bigcup_{i \in \calI'} \calE_i) > (2r - 1)(|\calI'| - 1)$ for every $\calI' \subseteq \calI$, then there is a family of hyperedges $\{e_i\}_{i \in \calI}$ such that $e_i \subseteq \calE_i$ and $e_i \cap e_j = \emptyset$ for every $i \neq j \in \calI$.
\end{theorem}

\subsection{A Ramsey-type lemma}

Next, let $K^m_k$ denote the complete $k$-partite graph with parts of size
$m$. The proof of Lemma \ref{lemma:smallset-cover} relies on the following
auxiliary result:

\begin{lemma} \label{lemma:gendfs}
  Let $m\geq 1$ and $k\geq 2$ be integers, and assume that $G$ is a graph
  whose complement is $K^m_k$-free. Then $G$ contains
  a collection of at most $k-1$ vertex-disjoint cycles covering all but
  at most $2k^2 m+k^3$ vertices.
\end{lemma}

The case $k=2$ of Lemma \ref{lemma:gendfs} states that if a graph does not have a large bipartite hole then it contains a large cycle. This was proved with a slightly
smaller leftover by Krivelevich and Sudakov \cite{ks2012phasetransition}. We
remark that the number of cycles given by the lemma is best possible if one
wants a leftover that can be bounded by a function of $m$ and $k$. This can be
seen by considering the disjoint union of $k-1$ cliques of size $n/(k-1)$: the
complement of such a graph does not contain a $K_k^m$ for any $m\geq 1$ and yet
any collection of $k-2$ cycles must necessarily leave $n/(k-1)$ vertices
uncovered. Although it does not matter for the present paper, it would be
interesting to see how much the size of the leftover in Lemma
\ref{lemma:gendfs} can be reduced.

We start the proof of Lemma \ref{lemma:gendfs} with the following claim, which is
essentially already contained in \cite{ks2012phasetransition}, included here
for completeness:

\begin{claim}\label{claim:dfs}
  Let $G$ be a graph with at least $m\geq 1$ vertices. Then there is
  a (possibly empty) path $P$ in $G$ and subsets $D,U\subseteq V(G)$ such that
  $|D|=m$, $e(D,U) = 0$, and
  \[ V(G) = D\cupdot V(P)\cupdot U. \]
\end{claim}

\begin{proof}
  Let $G=(V,E)$ where $|V|\geq m$.
  To prove the claim, we analyze the depth-first-search (DFS)
  algorithm when run on $G$. The state at the $i$-th step of this
  algorithm can be described by disjoint sets $D_i,U_i \subseteq V$
  and a possibly empty path $P_i$ in $G$. The set $D_i$ contains the
  \emph{discarded} vertices from which the DFS has already back-tracked. The
  set $U_i$
  represents the set of \emph{unexplored} vertices, that is, of those vertices
  that have not yet been visited by the DFS algorithm. 
  The path $P_i$ contains the vertices that have been visited, but are not discarded yet.
  The initial state of the algorithm is $(D_0,P_0,U_0) =
  (\emptyset,(),V)$: all vertices are unexplored.
  Given the state $(D_i, P_i,U_i)$ at the $i$-th step, the next state
  $(D_{i+1}, P_{i+1}, U_{i+1})$ is obtained using the following rules:
  \begin{enumerate}
    \item (\textsc{Terminate}) If $P_i = ()$ and $U_i = \emptyset$:  \\
    The algorithm terminates (there is no next state).

  \item (\textsc{Restart}) If $P_i = ()$ and  $U_i \neq \emptyset$: \\
    Choose an arbitrary vertex $w\in U_i$ and set
      $(D_{i+1},P_{i+1},U_{i+1}) = (D_i, (w), U_i \setminus \{w\})$.
   
    \item (\textsc{Explore}) If $P_i = (v_1, \dots, v_\ell)$ is non-empty and
    $U_i\cap N(v_\ell)\neq\emptyset$:\\
      Choose an arbitrary vertex $w\in U_i\cap N(v_\ell)$;\\
      the next state is
      $(D_i,(v_1, \dots, v_{\ell}, w),  U_i \setminus \{w\})$.
  \item (\textsc{Back-track}) If $P_i = (v_1, \dots, v_\ell)$ is non-empty but
    $U_i\cap N(v_\ell)=\emptyset$:\\
  The next state is $(D_i \cup \{v_\ell\}, (v_1, \dots, v_{\ell-1}), U_i)$.
  \end{enumerate}

  Note that at every step of the algorithm, the sets $D_i, V(P_i), U_i$ form a
  partition of $V$. It is also easy to see that the algorithm eventually
  terminates -- indeed, the value $|D_i|-|U_i|$ increases by exactly one in
  each step and is capped at $|V|$. As the initial state is $(D_0,P_0,U_0) =
  (\emptyset, (), V)$ and the terminal state is $(V,(),\emptyset)$, and because in
  every round, $|D_i|$ can increase only by at most $1$, we see that there
  is a state $(D_i,P_i,U_i)$ where $|D_i|=m$ ($\leq |V|$).

  The required path is then $P=P_i$ and the required sets are $D=D_i$ and $U=U_i$.
  Note that since vertices are moved to $D_i$ only if they no longer
  have any neighbors in $U_i$, we have $e(D_i,U_i)=0$ for all $i\geq
  0$ and in particular $e(D,U) = 0$.
\end{proof}

It will be convenient to first prove Lemma \ref{lemma:gendfs} under the
additional assumption that the given graph $G$ contains a Hamiltonian path.

\begin{claim} \label{claim:dfs_hamil}
  Let $m\geq 1$ and $k \ge 2$ be integers, and let $G$ be a graph whose
    complement is $K_k^m$-free. If $G$ contains a Hamiltonian
    path, then $G$ contains a collection of at most $k-1$ vertex-disjoint cycles
    covering all but at most $km$ vertices of $G$.
\end{claim}
\begin{proof}
  We prove the claim by induction on $k$. Let $n=|V(G)|$. We may assume $n > km$, as
  otherwise the statement is trivially satisfied. Let $P = (v_1, \dots, v_n)$
  be a Hamiltonian path in $G$. 

  In the base case $k = 2$, let $S_1 = \{v_1, \dots, v_m\}$ and $S_2 = \{v_{n -
  m + 1}, \dots, v_n\}$ be the sets containing the first and the last $m$
  elements of the path $P$, respectively. 
  As $n>2m$, the sets $S_1$ and $S_2$ are disjoint.
  Since the complement of $G$ does not
  contain a complete bipartite graph $K_2^m$, there must be an edge between
  $S_1$ and $S_2$. Together with $P$, such an edge forms a cycle containing all
  but at most $2m$ vertices.

  Let us now assume that the claim holds for all $k'< k$, for some $k \ge 3$.
  Consider the set $S_1 = \{v_1, \dots, v_m\}$ consisting of the first $m$
  vertices of the path $P$. Let $m< i\leq n$ be the largest index $i$ such that
  $v_i$ has a neighbor in $S_1$, and let $S=\{v_1,\dots,v_i\}$.
  Then $G[S]$ clearly contains a cycle that covers all but at most $m$ vertices from $S$.
  Moreover, there is no edge between $S_1$ and $S_2 = V(P)\setminus S= \{v_{i+
  1}, \dots, v_n\}$, and so, as the complement of $G$ is $K_k^m$-free, the
  complement of
  $G[S_2]$ is $K_{k-1}^m$-free. By induction $G[S_2]$ contains $k-2$
  vertex-disjoint cycles that cover all but at most $(k-1)m$ vertices from
  $S_2$. In total, there are $k-1$ vertex-disjoint cycles that cover all but
  at most $km$ vertices from $S\cup S_2=V(G)$.
\end{proof}

Next, we have the following simple observation:

\begin{claim}\label{claim:twoedges}
  Let $m\geq 1$ and $k \ge 2$ be integers, and let $G$ be a graph whose
  complement is $K_k^m$-free. Let $S_1,\dots,S_k\subseteq V(G)$ be disjoint sets
  of size at least $m+k-1$. Then for some $i\ne j$ there are two disjoint edges between $S_i$ and $S_j$.
\end{claim}
\begin{proof}
  The fact that the complement of $G$ does not contain a $K_k^m$ means that for every choice of
  disjoint subsets $S_1',\dots,S_k'$ of size at least $m$ each, we can find an edge going
  from $S_i'$ to $S_j'$ for some $i\ne j$.
  So let us start with $S_1,\dots,S_k$ of size at least $m+k-1$, find such an edge, and remove its endpoints
  from the sets $S_i$ and $S_j$. As long as the remaining sets (with the endpoints removed) still have size
  at least $m$, we can repeat this procedure. Eventually some $S_i$ will only have $m-1$ vertices left.
  This means that we have removed $k$ disjoint edges touching $S_i$, each going to some $S_j$ with $i\ne j$.
  Then by the pigeonhole principle, two of these edges must go to the same $S_j$.
\end{proof}

\begin{proof}[Proof of Lemma~\ref{lemma:gendfs}]
  Fix a graph $G=(V,E)$ on $n=|V|$ vertices and suppose that the complement of
  $G$ does not contain a copy of $K_k^m$.
  We start by proving a slightly weaker statement:
  \begin{equation}\label{eq:square}
    \text{$G$ contains $\tbinom{k}{2}$ vertex-disjoint cycles covering
    all but at most $k^2m$ vertices.}
  \end{equation}
  The proof of \eqref{eq:square} is by induction on $k$. 
  We first prove the base case $k=2$. By
  Claim~\ref{claim:dfs}, we can find a path $P$ and sets $D,U\subseteq V$ such that
  $|D|=m$, $e(D,U)=0$, and $V=D\cupdot V(P)\cupdot U$. Since
  the complement of $G$ is $K_2^m$-free, it follows that $|U|<m$, and so
  $|V(P)|\geq n-|D|-|U|> n-2m$.
  Now $G[V(P)]$ is a graph that contains a Hamiltonian path, and whose complement is $K_2^m$-free. 
  Thus, by Claim \ref{claim:dfs_hamil}, it contains a cycle covering all but
  at most $2m$ vertices of $G[V(P)]$, i.e., all but at most $4m$ vertices of $G$.
  This completes the proof for $k=2$.

  For the induction step, assume that $k\geq 3$ and that \eqref{eq:square} holds for
  $k-1$. Using Claim~\ref{claim:dfs} we find a path $P$
  and sets $D,U$ such that $|D|=m$, $e(D,U)=0$, and $V=D\cupdot
  V(P)\cupdot U$. As the complement of $G$ is $K_k^m$-free, and since there are
  no edges going between $D$ and $U$, the complement of $G[U]$ is
  $K_{k-1}^m$-free. By induction, there is a collection of $\binom{k-1}{2}$
  vertex-disjoint cycles in $G[U]$ that covers all but at most
  $(k-1)^2m$ vertices from
  $U$. Moreover, as $G[V(P)]$ is a graph with a Hamiltonian path whose complement is
  $K_{k}^m$-free, we can use Claim \ref{claim:dfs_hamil} to obtain $k-1$
  vertex-disjoint cycles in $G[V(P)]$ covering all but at most $km$ vertices of
  $G[V(P)]$. In total, we have a collection of $k-1+\binom{k-1}{2}= \binom{k}{2}$
  vertex-disjoint cycles in $G$ covering all but at most $(k-1)^2m+km+m\le
  k^2m$ vertices, establishing \eqref{eq:square}.
  
  All in all, we have collection of $\binom{k}{2}$ vertex-disjoint cycles
  covering all but at most $k^2m$ vertices.
  We reduce the size of this collection as follows.
  If the collection contains a cycle shorter than $m+k-1$, we
  remove this cycle from the collection, increasing the number of uncovered
  vertices by at most $m+k-1$. Otherwise, if there are at least $k$
  cycles $C_1,\dots,C_k$ of length at least $m+k-1$ left, then we apply Claim \ref{claim:twoedges} to 
  a set $S_i$ of $m+k-1$ consecutive vertices from each cycle $C_i$. 
  The claim provides two edges that merge two of these cycles into a single
  one, while increasing the size of the leftover by at most $2(m+k-1)$.
  If we repeat this until the collection contains at most $k-1$ cycles, then
  we end up with $k-1$ vertex-disjoint cycles covering all except at most
  $k^2m + \binom{k}{2}2(m+k-1) \leq 2k^2m+k^3$ vertices, as required.
\end{proof}

\subsection{Properties of random graphs}

In this section we collect some properties of random graphs that are used
throughout the proof. The following Chernoff-type bounds on the tails of the
binomial distribution are used throughout.

\begin{lemma}[{\cite[Theorem 2.1]{janson2011random}}]\label{lem:che}
  Let $X \sim \emph{\text{Bin}}(n,p)$ be a binomial random variable. Then
  \begin{itemize}
    \item $\Pr\left[X<(1-a)np\right]<e^{-a^2np/2}$ for every $a>0$, and
    \item $\Pr\left[X>(1+a)np\right]<e^{-a^2np/3}$ for every $0<a<3/2$.
  \end{itemize}
\end{lemma}

The next result says that w.h.p a random graph contains approximately the expected number of edges between any two sufficiently large vertex sets.

\begin{lemma} \label{lemma:densityXY}
  Fix $0< \alpha, \beta < 1$
  and let $C=6/(\alpha^2\beta)$ and
  $D=9/\alpha^2$. Then for every $p = p(n) \in (0,1)$, the random
  graph $G \sim \Gnp$ satisfies the following property w.h.p: For any two disjoint
  subsets $X,Y\subseteq V(G)$ satisfying either of
  \begin{enumerate}
    \item $|X|, |Y| \ge D(\log n)/ p$, or
    \item $|X| \ge C / p$ and $|Y| \ge \beta n$,
  \end{enumerate}
  we have
  \[
    e(X,Y) \in (1 \pm \alpha) |X| |Y| p.
  \]
\end{lemma}
\begin{proof}
  For fixed sets $X$ and $Y$, the quantity $e(X,Y)$ follows the binomial
  distribution $\Bin(|X||Y|,p)$, so we can apply Lemma \ref{lem:che} to get 
  \[\Pr\big[e(X, Y) \notin (1 \pm \alpha)|X||Y|p\big] \le 2e^{-\alpha^2|X||Y|p/3}.\]
  Hence the probability that there exist sets $X,Y\subseteq V(G)$ such that
  $|X|\ge |Y| \ge  D(\log n)/p$ and $e(X,Y)\notin (1\pm\alpha)|X||Y|p$ is at most
 \begin{align*}
    \sum_{\frac{D \log n }{p}\le y\le x\le  n}
    \binom{n}{x} \binom{n}{y} 2e^{- \alpha^2 x y p/3}
    &\le \sum_{\frac{D \log n }{p}\le y\le x\le n}
    e^{2x \log n} e^{-x \cdot \alpha^2 D( \log n)/3} \\
    &\le  \sum_{\frac{D \log n }{p}\le y\le x\le n} e^{-x\log n} \le n^2
    e^{-\log^2 n} = o(1),
  \end{align*}
  using $D= 9/\alpha^2>1$. Similarly, the probability that there are subsets
  $X,Y\subseteq V(G)$ such that$|X|\ge C/p$, $|Y|\ge \beta n$, and
  $e(X,Y) \notin (1 \pm \alpha) |X| |Y| p$
  is at most
  \[
    \sum_{x = C / p}^n \sum_{y = \beta n}^n  \binom{n}{x} \binom{n}{y} 2e^{- \alpha^2 x y p/3} \le 
  \sum_{x = C / p}^n  \sum_{y = \beta n}^n 2^{2n} e^{-\alpha^2 C \beta n/3} = o(1),
\]
  for $C\ge 6/\alpha^2\beta$.
\end{proof}

The following lemma studies how the common neighborhoods of given vertex pairs
intersect an arbitrary set.

\begin{lemma} \label{lemma:density_triples}
  For every $p = p(n) \in (0,1)$, the random graph $G \sim \Gnp$ satisfies
  the following property w.h.p: For every family $\calL$ of $\ell$ disjoint
  pairs of vertices, and for every set $Y$ of $3\ell$ vertices that is disjoint
  from
  each pair in $\calL$, we have
  \[
    \sum_{\{v,w\} \in \calL} |N(v, Y) \cap N(w, Y)| \le \begin{cases}
      72 \ell \log n &\text{if } \ell \le 6 \log n / p^2, \\
                2\ell|Y| p^2 &\text{otherwise.}
    \end{cases}
  \]
\end{lemma}

\begin{proof}
  Fix sets $\calL$ and $Y$ as in the statement of the lemma. Since the
  pairs in $\calL$ are pairwise disjoint and disjoint from $Y$, the
  sum 
  \[
    Q(\calL, Y) = \sum_{\{v,w\} \in \calL} |N(v, Y) \cap N(w, Y)|
  \]
  has the same distribution as a sum of $\ell |Y|=3\ell^2$ independent
  Bernoulli random variables with probability $p^2$. In other words, $Q(\calL,
  Y)\sim \Bin(3\ell^2,p^2)$. Therefore, for $\ell \le 6 \log n / p^2$ we have
  \begin{align*}
    \Pr [ Q(\calL,Y) > 72\ell \log n] &\le  \binom{3\ell^2}{72 \ell\log n}
    (p^2)^{72\ell\log n} \le \Big( \frac{e\cdot  3\ell^2
    p^2}{72\ell \log n }\Big)^{72\ell\log n} \\
    &\le  (e/4
    )^{72\ell \log n} \le e^{-6\ell \log n},
  \end{align*}
  where the last inequality follows from $(e/4)^{12} \le 1/e$. Otherwise, if
  $\ell > 6 \log n / p^2$ then we can apply the Chernoff bound (Lemma
  \ref{lem:che}) to get
  \[
    \Pr[ Q(\calL,Y) > 2\ell|Y|p^2] \le e^{-\ell|Y|p^2/3} = e^{-\ell^2p^2} \le
    e^{- 6 \ell\log n }.
  \]
  Taking a union-bound over all choices of $\calL$ and $Y$, we obtain that the
  probability that for some $\calL$ and $Y$ the desired upper bound does not
  hold is at most \[
    \sum_{\ell = 1}^{n} \binom{n^2}{\ell} \binom{n}{3\ell} e^{-6\ell \log n}  \le
    \sum_{\ell = 1}^{n} n^{2\ell + 3\ell} e^{-6 \ell \log n} = \sum_{\ell=1}^n
    n^{-\ell} \to 0,
  \]
  completing the proof.
\end{proof}

Our last tool shows that the common neighborhoods of not too many distinct small sets are close to disjoint.

\begin{lemma}\label{lemma:tuples_expand}
  Let $r \ge 2$ be a fixed integer and let $\teps=\teps(n)
  \in(0,1)$. Let $p=p(n)\in (0,1)$ and let $L\subseteq[n]$ be a fixed set of at least
  $50 r \log n/(\teps p^r)$ vertices. Then $G\sim\Gnp$ w.h.p satisfies the
  following property: For every family of at
  most $\teps/p$ different sets $X_1,\dots,X_t\subseteq [n]\setminus L$ of size $r$,
  we have
  \[ \Big| \bigcup_{i=1}^t \CN(X_i,L) \Big| \in (1 \pm \sqrt{\teps}) t |L| p^r. \]
\end{lemma}
\begin{proof}
  Let $X_1,\dots,X_t$ be a fixed family of
  $t\leq \teps/p$ different
  $r$-sets
  contained in $[n]\setminus L$.
  We consider the random set $W=
  \bigcup_{i=1}^t \CN(X_i,L)$. Note that $W$ contains each vertex of $L$
  independently with the same probability, so $|W|\sim \Bin(|L|,q)$ for some
  probability $q$. We will use the inclusion-exclusion principle (Bonferroni's
  inequality) to estimate the expectation of $|W|$, and then apply Chernoff
  bounds to get concentration.

  Let $A_i=\CN(X_i,L)$. Then, according to Bonferroni's inequality, we have
  \[
   \sum_{i=1}^t |A_i|- \sum_{1\le i<j\le t}  |A_i\cap A_j|\le
   \Big|\bigcup_{i=1}^t A_i\Big| = |W|  \le \sum_{i=1}^t |A_i|.
  \]
  Here $\Exp[|A_i|]=\Exp[|\CN(X_i,L)|]=|L|p^r$ and $\Exp[|A_i\cap A_j|] =
  \Exp[|\CN(X_i\cup X_j,L)|] = |L|p^{|X_i\cup X_j|}$. In particular,
  $\Exp[|A_i\cap A_j|]\le |L|p^{r+1}$.
  Thus, on the one hand we have
  \[ \Exp[|W|] \leq \sum_{i=1}^t \Exp[|A_i|] = t |L|p^r,\]
  and on the other hand,
  \[ \Exp[|W|] \geq \sum_{i=1}^t \Exp[|A_i|] - \sum_{1\leq i<j\leq t}
  \Exp[|A_i\cap A_j|] \geq t |L|p^r - \binom{t}2|L|p^{r+1}
  \geq t|L|p^r - t\teps |L|p^r/2,\]
  using $t\leq \teps/p$ in the last inequality. Hence
  \[  (1-\teps/2)t|L|p^r \leq \Exp[|W|] \leq t|L|p^r, \]
  so in particular, $q\ge (1-\teps/2)tp^r \ge tp^r/2$. Then by a Chernoff bound, i.e., Lemma~\ref{lem:che} with $a=\sqrt{\teps}/2$, we get
  \[ \Pr[|W|\not\in (1\pm \sqrt{\teps})t|L|p^r]\leq 2e^{-\teps t|L|p^r/24}
  \leq e^{-2tr\log n}, \]
  using $|L|\geq 50r\log n/(\teps p^r)$ in the last inequality.
  Finally, a union-bound over all choices of $X_1,\dots,X_t$ shows that the property in the lemma fails with probability bounded by
  \[ \sum_{t=1}^{\teps/p} \binom{n}{r}^t  e^{- 2tr \log n} \le
  \sum_{t=1}^{\infty}e^{-tr\log n} \le \sum_{t=1}^{\infty} n^{-t} \to 0.\qedhere
  \]
\end{proof}

\section{Approximate covering -- proof of Lemma \ref{lemma:smallset-cover}} \label{sec:almost_covering}

Let $W$ be a fixed set of at most $n/(100r)^4$ vertices and let $G\sim \Gnp$, where $p=p(n)\gg (\log n/n)^{1/2}$. Consider some coloring of
the edges of $G$ with $r$ colors. Our goal is to find $3r^2$ vertex-disjoint
monochromatic cycles that cover all but at most $K/p$ vertices
of $W$, for some sufficiently large constant $K=K(r)$.

Let $U=V(G)\setminus W$. It is convenient to work with different colors
separately, so we partition $W$
into $r$ sets \[W= W_1 \cupdot \dots \cupdot W_r\] such that
$v\in W_i$ if the most commonly used color on edges between $v$ and $U$
is $i$:
\[ v\in W_i \implies |N_i(v,U)| = \max_{j\in [r]}{|N_j(v,U)|}. \]
Recall here that $N_i(v,U)$ is the set of vertices in $U$ joined to $v$ by an
edge of color $i$.

Next, for every $i\in [r]$, we define an auxiliary graph $H_i$ on $W_i$
where two vertices $v, w \in W_i$ are connected by an edge if and only if
\[
	|N_i(v, U) \cap N_i(w, U)| \ge \frac{n p^2}{(50r)^4}. 
\]

The main idea is to apply Lemma \ref{lemma:gendfs} on $H_i$ to find a small
collection of ``auxiliary cycles'' covering most of $W_i$. Then we will use
Hall's condition to turn each such auxiliary cycle into a
cycle in $G_i$ (the subgraph of $G$ of edges in color $i$) that covers
the same vertices in $W_i$.
We thus have two claims:

\begin{claim}\label{cl:1}
  Let $c=384r$. Then w.h.p each auxiliary graph $H_i$ contains $3r-1$
  vertex-disjoint cycles covering all but at most $18r^2c/p +(3r)^3$
  vertices of $H_i$.
\end{claim}

\begin{claim}\label{cl:2}
  The following holds w.h.p: For every choice of $k_i\leq 3r-1$ vertex-disjoint
  cycles $C_1^{i},\dots,C_{k_i}^{i}$ in each auxiliary graph $H_i$, the
  graph $G$ contains $k_1+\dots+k_r$ vertex-disjoint monochromatic cycles
  covering all the sets $V(C_j^{i})$.
\end{claim}

From these two claims, Lemma \ref{lemma:smallset-cover} follows immediately
with $K=r\cdot 18r^2\cdot 384 r+27r^4 \le (20r)^4$. It remains to prove the two claims.

\begin{proof}[Proof of Claim \ref{cl:1}]
  In light of Lemma \ref{lemma:gendfs}, it is enough to show that the complement
  of each $H_i$ is $K^{c/p}_{3r}$-free.
  There is an easy intuition as to why this is the case: in $G$, each vertex of $W$ is
  adjacent to about $p|U|$ vertices in $U$, so we expect a subset of size
  $c/p$ to expand to almost the whole $U$. As a vertex from $W_i$ has at least $p|U|/r$ neighbors in $U$ in color $i$, this suggests that a subset of $c/p$
  vertices of $W_i$ should have at least $|U|/r$ neighbors in this color $i$. But then if we take a bit more than $r$ sets of this size, then their
  $i$-colored neighborhoods overlap significantly, so there should be two
  sets with linearly many common neighbors in color $i$. This readily implies that some two vertices in $H_i$ are joined by an edge. We will now make this argument precise.

  First, w.h.p $G$ satisfies the properties of Lemma \ref{lemma:densityXY} with
  $\alpha=1/4$ and $\beta=1/(4r)$. Let $c =
  C_{\ref{lemma:densityXY}}(\frac{1}{4}, \frac{1}{4r})=384r$ be the
  corresponding constant given by Lemma \ref{lemma:densityXY}. Then Lemma
  \ref{lemma:densityXY} states that w.h.p any two disjoint subsets
  $X,Y\subseteq V(G)$ such that $|X|\geq c/p$ and $|Y|\geq n/(4r)$ satisfy
  \begin{equation}\label{eq:densityXY}
    e(X,Y) \leq \frac54 |X||Y|p.
  \end{equation}
  Next, every vertex $v\in W$ has $e(v,U)\sim \Bin(|U|,p)$
  neighbors in $U$, thus Lemma \ref{lem:che} shows that the
  probability of $e(v,U)\notin (1 \pm r^{-2})p|U|$ is at most
  $2e^{-|U|p/(3r^4)} = e^{-\Omega(np)}$, using $|U|\geq n/2$. As 
  $p>1/\sqrt{n}$,  a union bound over all vertices of
  $W$ gives that w.h.p $e(v, U) \in (1 \pm r^{-2}) p|U|$ for every  $v \in W$. 
  By the definition of $W_i$, it follows that w.h.p
  we have \begin{equation}\label{eq:deg}
    e_i(v, U) \ge e(v,U)/r\geq  p|U| / (r + 1)\quad\text{for every $v\in W_i$}.
  \end{equation}

  In the following, fix some $i\in [r]$. We will show that properties
  \eqref{eq:densityXY} and \eqref{eq:deg} already imply that the
  complement of $H_i$ is $K_{3r}^{c/p}$-free. In other words, we show that
  if $X_1,\dots,X_{3r}\subseteq W_i$ are disjoint sets of size $c/p$, then
  there exist $j\neq j'$ such that $e_{H_i}(X_j,X_{j'})> 0$.

  Fix any such sets, and let $Y_j = N_i(X_j, U)$ denote the set of vertices in
  $U$ that have a neighbor in $X_j$
  in color $i$. We first show that \begin{equation}\label{eq:expando}
  |Y_j| > |U| / (2r) \quad\text{for every $j\in[3r]$}.
  \end{equation}
  First, by \eqref{eq:deg} we have
  \[ e_i(X_j,Y_j) = e_i(X_j,U) \geq |X_j||U|p/(r+1). \]
  Suppose that $|Y_j| \le |U| / (2r)$ and choose an arbitrary set $Y_j \subseteq Y_j' \subseteq U$  of size $|U| / (2r) \ge n / (4r)$. Then by \eqref{eq:densityXY}, we have
  \[
    e_i(X_j,Y_j) \leq e(X_j, Y'_j) \le \frac{5}{4} |X_j| |Y'_j| p = \frac{5}{8r} |X_j|
      |U| p < |X_j||U| p / (r + 1), 
  \]
  which is a contradiction. This establishes \eqref{eq:expando}.

  Now we can bound how much these colored neighborhoods intersect using Bonferroni's
  inequality:
  \[
  \sum_{1\le j < j' \le 3r} |Y_j \cap Y_{j'}| \ge \Big( \sum_{j = 1}^{3r} |Y_j|
  \Big) - |U| \ge |U|/2,
  \]
  using \eqref{eq:expando} for the last inequality.
  In particular, we must have $|Y_j \cap Y_{j'}| \ge |U| /
  (3r)^2$ for some $j\neq j'$. Note that $Y_{j} \cap Y_{j'}$ can also be
  written as the union of the common neighborhoods $N_i(v,U) \cap N_i(w,U)$
  over all pairs of vertices $v\in X_j$ and $w\in X_{j'}$. But then for some
  choice of $v$ and $w$, the size of this is at least the average:
  \[
  |N_i(v, U) \cap N_i(w, U)| \ge \frac{|U|}{(3r)^2} \cdot
  \frac{1}{|X_j||X_{j'}|} = \frac{|U| p^2}{(3r)^2 c^2} \ge \frac{n
  p^2}{(50r)^4}
  \]
  using $|U|\ge n/2$ and $c=384r$. Thus $v$ and $w$ are connected by an edge in $H_i$ 
  and so $e_{H_i}(X_j,X_{j'})>0$, which is what we set
  out to prove.
\end{proof}

\begin{proof}[Proof of Claim \ref{cl:2}]
  First note that w.h.p $G$ satisfies the conclusion of Lemma
  \ref{lemma:density_triples}, which says that for every family $\calL$ of
  $\ell$ disjoint pairs of vertices, and for every set $Y$ of $3\ell$ vertices
  that is disjoint from each pair in $\calL$, we have
  \begin{equation}\label{eq:lem39}
    \sum_{\{v,w\} \in \calL} |N(v, Y) \cap N(w, Y)| \le \begin{cases}
      72 \ell \log n, &\text{if } \ell \le 6 \log n / p^2 \\
      2\ell|Y| p^2, &\text{otherwise.}
    \end{cases}
  \end{equation}
  This property will be enough to deduce the claim.

  Let $\calE_i= \bigcup_{j \in [k_i]} E(C^i_j)$ be the edge set of the given
  cycles in the auxiliary graph $H_i$, and let
  $\calE=\bigcup_{i \in [r] } \calE_i$. We define an
  auxiliary bipartite graph $B$ with parts  $U$ and $\calE$
  where an edge $vw \in \calE_i$
  is joined to a vertex $u\in U$ if and only if $u \in N_i(v,U) \cap N_i(w,U)$. We
  will use Hall's condition to show $B$ contains a matching covering $\calE$.
  In other words, we will show that there exists an injection $f\colon \calE\to
  U$ such that for every $vw\in \calE_i$, the vertex $f(vw)\in U$ is a common neighbor
  of both $v$ and $w$ in color $i$. This will immediately imply the statement
  of Claim \ref{cl:2}, because we can then convert every cycle $C_j^i$ in $H_i$
  into a monochromatic cycle in $G$ by replacing each edge $vw$ of $C_j^i$ by
  an $i$-colored path $(v,f(vw),w)$. The injectivity of $f$ ensures that the
  cycles we get are vertex-disjoint.

  To verify Hall's condition, we need to show that for every subset
  $\calL \subseteq \calE$ we have $|N_B(\calL)| \geq  |\calL|$. We instead
  prove the somewhat different statement that for every subset $\calL \subseteq
  \calE$ consisting of \emph{pairwise disjoint} edges in $\calE$ we have
  $|N_B(\calL)| \geq 3 |\calL|$. This second statement actually implies the first:
  as $\calE$ is
  a disjoint union of cycles, every set of edges $\calL' \subseteq\calE$
  contains a subset $\calL \subseteq \calL'$ of size at least $|\calL'|/3$ such
  that the edges in $\calL$ are vertex-disjoint, and so
  \[
   |N_B(\calL')| \ge |N_B(\calL)| \ge 3 |\calL| \ge |\calL'|.
  \]

  So take any set of pairwise disjoint edges $\calL \subseteq \calE$
  and suppose for contradiction that $N_B(\calL)$ is properly
  contained in some set $Y\subseteq U$ of size $3|\calL|$.
  Recall that every vertex pair $vw\in \calL$ is also an
  auxiliary edge in some $H_i$, so $v$ and $w$ have at least $np^2/(50r)^4$
  common neighbors in $U$ in color $i$. By definition, these common neighbors
  are also neighbors of the pair $vw$ in $B$, so they are contained in $Y$. Hence
  \[ \frac{n p^2 |\calL|}{(50r)^4}\leq \sum_{vw\in \calL} |N(v,Y)\cap N(w,Y)|. \]
  Note that since $|Y|=3|\calL|$, we can apply \eqref{eq:lem39} to the sum on
  the right-hand side.
  There are now two cases, depending on the size of $\calL$.

  If $|\calL| \le 6 \log n / p^2$, then by \eqref{eq:lem39}
  \[ \frac{np^2|\calL|}{(50r)^4} \leq \sum_{vw\in \calL} |N(v,Y)\cap N(w,Y)|\leq
  72|\calL|\log n, \]
  contradicting our assumption that $np^2 \gg \log n$.

  If $|\calL| > 6 \log n / p^2$, then by \eqref{eq:lem39}
  \[ \frac{np^2|\calL|}{(50r)^4} \leq \sum_{vw\in \calL} |N(v,Y)\cap N(w,Y)|\leq
  2|\calL||Y|p^2, \]
  and thus $|Y| \geq n/(2\cdot (50r)^4) > 3n/(100r)^4\geq 3|W|\geq 3|\calL|$,
  which contradicts the assumption that $|Y| = 3|\calL|$. This concludes the
  proof of the claim.
\end{proof}

\section{Covering a set of size $O(1/p)$ -- proof of Lemma \ref{lemma:tinyset-cover}} \label{sec:covering_small_set}

Let $r, \eps , K > 0$ be constants, where $r\geq 2$ is an integer. Consider a  subset $W \subseteq [n]$ of size at most $n/2$ and let $G \sim \Gnp$ for
$p =p(n) > n^{-1/r +\eps}$. We need to show that w.h.p every subset
$Q\subseteq W$ of size at most $K/p$ can be covered using at most $400r^4
\log(4r^2)$ monochromatic cycles.

\subsection{Proof overview} \label{sec:overview}

The main idea of the proof is the following: We 
define an edge-colored auxiliary graph $H$ on the vertex set $W$, where two
vertices $v,w\in W$ are joined by an edge of color $i$ if they are ``robustly
connected'' by monochromatic paths of color $i$, whose interior vertices belong to
$U = V(G) \setminus W$. This auxiliary graph should have two properties.
First, we want the independence number of $H$
to be bounded by a function of $r$, as this will imply, via the result of
Sárközy (Theorem \ref{thm:indep}), that every induced subgraph $H[Q]$ can be
covered by a number of monochromatic cycles in $H[Q]$ that depends only on $r$.
Second, we want the notion of ``robustly connected'' to be sufficiently strong
to allow us to convert such a cover of $H[Q]$ by auxiliary
monochromatic cycles into a cover of $Q$ by monochromatic cycles in $G$,
at least if $|Q|\leq K/p$.

It is instructive to note that this task would be significantly easier for $p\gg
n^{-1/(r+1)}$. In this case we could define $H$ by saying that $v,w\in W$ are
joined by an edge of color $i$ if there are many (i.e., $\Omega(np^{r+1})$)
vertices in the $i$-colored common neighborhood $N_i(v,U)\cap N_i(w,U)$. It is
not hard to see that if $p\gg (\log n/n)^{1/(r+1)}$, then this graph has independence
number at most $r$. Indeed, the high density implies that every set $X\subseteq
W$ of $r+1$ vertices has a large common neighborhood $\CN(X,U)$ in $U$ (of size
$\approx np^{r+1}$). Of course, for every vertex in $\CN(X,U)$, at least two of
the edges coming from $X$ must have the same color (by the pigeonhole
principle). This in turn implies that some two vertices in $X$ will have a
large common neighborhood in the same color, so $H[X]$ contains an
edge. We could then apply Theorem \ref{thm:indep} to cover $H[Q]$ with few
disjoint monochromatic cycles, and, as in the proof of Lemma
\ref{lemma:smallset-cover}, use Hall's condition to turn each of these
auxiliary cycles into a monochromatic cycle in $G$ by replacing each auxiliary
edge with a path of length 2 in $G$ in the same color.

Unfortunately, when $p$ is smaller than $n^{-1/(r+1)}$ a typical set of $r+1$
vertices does not have any common neighbor, and the graph $H$ as defined in
the previous paragraph might have unbounded independence number. We overcome this
issue by using slightly longer paths to create monochromatic cycles in $G$. So here the edges of $H$
will correspond to short monochromatic paths whose lengths are possibly greater than 2.
We now describe informally how we are going to do this, assuming for
simplicity that $r=2$, i.e., that there are only two colors, called red and
blue.

\medskip
To define $H$, consider an arbitrary set $\hX\subseteq W$ consisting of $3$
vertices $u,v,w$ (see Figure~\ref{figure}). Since we assume $p> n^{-1/r+\eps}=n^{-1/2+\eps}$, any two
vertices in $\hX$ will have $\Theta(np^2)$ common neighbors in $U$. Let $Z_1$
be the set of common neighbors of $u$ and $v$ in $U$ (and keep the vertex $w$
for later). If there are $\Omega(np^2)$ vertices $x\in Z_1$ that have an edge
to both $u$ and $v$ in the same color, then we add an edge in that color in $H$
between $u$ and $v$. However, it could be that most vertices in $Z_1$ are
connected to $u$ and $v$ by edges of different colors (i.e., there is a red edge to
$u$ and a blue edge to $v$, or vice-versa). In this case, we can find many
vertices in $Z_1$ which are of the ``majority color profile'', that is, we can
find a set $S_1\subseteq Z_1$ of $\Omega(np^2)$ common neighbors of $u$ and $v$
such that either every vertex $x\in S_1$ has a red edge to $u$ and a blue edge
to $v$, or every vertex $x\in S_1$ has a blue edge to $u$ and a red edge to
$v$. In any case, we can relabel
$\hX=\{v_\text{free},v_\text{red},v_\text{blue}\}$ such that every vertex in
$S_1$ is connected by a red edge to $v_\text{red}$ and by a blue edge to
$v_\text{blue}$ (and $v_\text{free}$ will just be the vertex $w$).

We now define the set $Z_2$ of all vertices in $U$ that have an edge to both 
$v_\text{free}$ and a vertex in $S_1$. The properties of $\Gnp$ will make
sure that this set is about as large as expected: $|Z_2|=\Omega(n^2p^4)$. If
there are $\Omega(n^2p^4)$ vertices in $Z_2$ that have an edge to both
$v_\text{free}$ and a vertex of $S_1$ in the same color -- say, both are blue
-- then we add a blue edge to $H$ between $v_\text{free}$ and
$v_\text{blue}$. And of course, if both edges are red, then we
add a red edge between $v_\text{free}$ and $v_\text{red}$. This way an
edge of color $i$ corresponds to many $i$-colored paths of length 3 between the
two involved vertices. However, as before, it could be that most vertices
in $Z_2$ are connected to $v_\text{free}$ and $S_1$ in both colors. Then,
again, there must be a majority color profile of such vertices, and we can find
a subset $S_2\subseteq Z_2$ of size $\Omega(n^2p^4)$ such that every vertex in
$S_2$ has either a blue edge to $v_\text{free}$ and a red edge to a vertex of
$S_1$, or the other way around. The important observation is that it is again 
possible to relabel the vertices $\hX = \{v_\text{free}, v_\text{red},v_\text{blue}\}$
such that every vertex in $S_2$ is now connected by a blue path to
$v_\text{blue}$ and by a red path to $v_\text{red}$ (for example, if every
vertex in $S_2$ has a blue edge to $S_1$ and a red edge to $v_\text{free}$, we
exchange the identities of $v_\text{red}$ and $v_\text{free}$).

If we continue like this, the following pattern emerges: starting from a set
$\hX$ of three vertices, either we are able to place an edge in $H$
between two vertices in $\hX$, or we get a sequence $S_1,S_2,S_3,\dots$ of
larger and larger sets ($|S_i|$ will be around $(np^2)^i$) such that every vertex in
$S_i$ is connected by a red path to some $v_\text{red}\in \hX$ and by a blue
path to some $v_\text{blue}\in \hX$. This statement is formalized in Claim
\ref{claim:cascade_tower} below.

Now take any set $X$ of 6 vertices, and split it into two sets $\hX$ and $\hX'$
of three vertices each. If $H$ contains an edge inside $\hX$ or $\hX'$, then
$X$ is not independent. Otherwise,  since $p>n^{-1/2+\eps}$, we see that after
about $m=1/\eps$ iterations of the above procedure, we reach a set $S_m$ from
$\hX$ and a set $S_m'$ from $\hX'$, both of size much larger than $(\log n)/p$,
such that every vertex in the set $S_m$ ($S_m'$) is connected in both colors to
the set $\hX$ ($\hX'$).
In $\Gnp$ there are many edges between any two sets of size larger than $(\log n)/p$ 
(see Lemma \ref{lemma:densityXY}), in particular, we can find
many (say) red edges between $S_m$ and $S_m'$. But then there is
a vertex in $\hX$ that is connected to a vertex in $\hX'$ by many red
paths, and we can add a red edge connecting these vertices to $H$.
This gives us an $H$ that has independence number at most 5
such that an $i$-colored edge in $H$ corresponds to many $i$-colored paths
between the two vertices in $G$.

The same general approach works if there are more than two colors. In the rest
of this section, we explain the details of the above outline to get a real proof, 
and show how to turn the auxiliary cycles in $H[Q]$ into monochromatic cycles in $G$.

\begin{figure}[t]
\begin{center}
\begin{tikzpicture}
	\tikzstyle{vertex}=[circle,draw=black,fill=black,inner sep=0,minimum size=4pt,text=white,font=\footnotesize]

	\draw[rounded corners=1mm,dotted] (0,0) rectangle (10,1);
	\draw[rounded corners=1mm,dotted] (0,1.5) rectangle (10,2.5);
	\draw[rounded corners=1mm,dotted] (0,3) rectangle (10,4);
	\draw[rounded corners=1mm,dotted] (0,4.5) rectangle (10,5.5);
	\node at (-.5,.5) {$L_0$};
	\node at (-.5,2) {$L_1$};
	\node at (-.5,3.5) {$L_2$};
	\node at (-.5,5) {$L_3$};

	\draw[red,pattern=vertical lines,pattern color=red] (3.03,2.1)--(1.5,.6)--(3.9,1.84)--(3.03,2.1);
	\draw[blue,fill=blue, fill opacity=.1] (3.1,1.84)--(5,.6)--(3.97,2.1)--(3.1,1.84);
	\draw[red,pattern=vertical lines,pattern color=red] (5.5,3.5)--(8.5,.6)--(7.5,3.5)--(5.5,3.5);
	\draw[blue,fill=blue, fill opacity=.1] (3.03,2.1)--(3.9,1.84)--(7.4,3.4)--(5.5,3.5)--cycle;
	\draw[red,pattern=vertical lines,pattern color=red] (5.5,3.5)--(7.4,3.6)--(4.5,5)--(1.65,4.9)--cycle;
	\draw[blue,fill=blue, fill opacity=.1] (1.5,5)--(1.5,.6)--(4.5,5)--(1.5,5);

	\node[vertex,label=below:{$u$}] (u) at (1.5,.6) {};
	\node[vertex,label=below:{$v$}] (v) at (5,.6) {};
	\node[vertex,label=below:{$w$}] (w) at (8.5,.6) {};

	\draw[fill=white] (3.5,2) ellipse (.5 and .25);
	\draw[fill=white] (6.5,3.5) ellipse (1 and .25);
	\draw[fill=white] (3,5) ellipse (1.5 and .25);
	\node at (3.5,2) {$S_1$};
	\node at (6.5,3.5) {$S_2$};
	\node at (3,5) {$S_3$};

\end{tikzpicture}
\end{center}
\caption{Building towers when $r=2$.} 
\label{figure}
\end{figure}
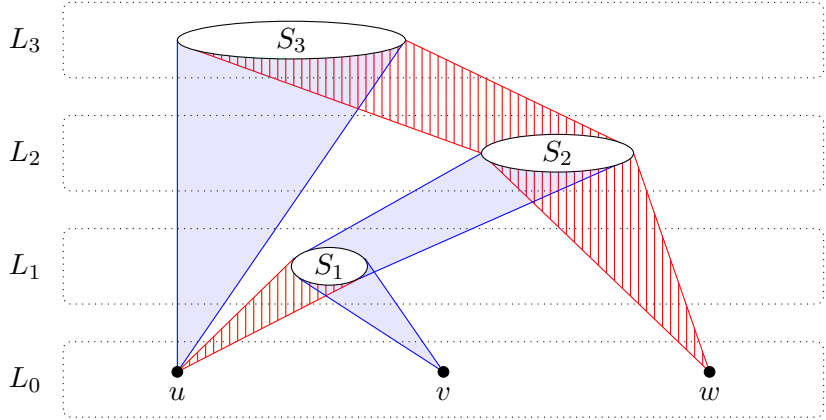

\subsection{Towers, cascades, and the auxiliary graph}
Without loss of generality, we may assume that $p=n^{-1/r+\eps}$ where $\eps = 1/qr$ for some large integer $q$ such that $r$ does not divide $q-1$. (So we have $p=n^{-(q-1)/qr}=n^{-(q-1)\eps}$.)

Let us fix a partition of $U=V(G)\setminus W$ into $1/\eps$ levels of the same size:
\[ U = L_1\cupdot \dots \cupdot L_{1/\eps}, \]
where $|L_k| = \eps|U|$. For notational convenience, we also define $L_0= W$.
As this partition does not depend on the choice of $Q \subseteq W$, we may fix it before exposing $G$. 
Set
\begin{equation}\label{eq:mu}
  \mu = \frac{|L_k| p^r} {2 r^r} = \Theta(n^{r \eps}),
\end{equation}
and let $m \in \mathbb{N}$ be such that $m-1 < \frac{q-1}{r} < m$. Note that this implies $m\le 1/\eps$, and 
\begin{equation} \label{eq:mu_cow_m}
  \mu^{m-1} = O(n^{-\eps} / p) \quad \text{ and } \quad \mu^m = \Omega(n^{\eps} / p).
\end{equation}
Indeed, $\mu^{m-1} \le \mu^{(q-2)/r} =O(n^{(q-1)\eps-\eps})$ and $\mu^{m} \ge \mu^{q/r} =\Omega(n^{(q-1)\eps+\eps})$.

Next, we describe the structures needed for the proof. The point of the argument we sketched in Section~\ref{sec:overview} was to find monochromatic paths from a vertex $v$ through a sequence of sets $(S_s,\dots,S_f)$ (see Figure~\ref{figure}). We will call such a monochromatic piece a \emph{tower} on $v$. Due to technical reasons, the formal definition below needs to include some auxiliary sets, as well.

\begin{definition}[Towers]\label{def:tower}
  Let $1\leq s\leq f\leq m$, let $i\in [r]$ be a color and $v\in L_0$. We call a
  sequence of sets $(S_{s-1},S_s, \dots, S_f)$ an \emph{$i$-tower on $v$}
  if there is a sequence of sets $T_s,\dots,T_f\subseteq L_0$ of size $r-1$
  such that the following properties hold:
   \begin{enumerate}[label=(T\arabic*)]
     \item\label{t1} $S_k\subseteq L_k$ and $|S_k| = \mu^k$ for all $k\in \{s-1,\dots,f\}$,
     \item\label{t2} $S_s\subseteq N_i(v) \cap N(S_{s-1})\cap \CN(T_s)$,
     \item\label{t3} $S_k\subseteq N_i(S_{k-1})\cap\CN(T_k)$ for all $k\in \{s+1,\dots,f\}$,
     \item\label{t4} $v\in T_s$ if $s>1$, and $S_0=\{v\}$ with $v\notin T_1$ if $s=1$.
  \end{enumerate}
  We say that $(T_s,\dots,T_f)$ is a \emph{witness sequence} for the tower $(S_{s-1},\dots,S_f)$.
\end{definition}

Note that in the case $r=2$ that we discussed in Section~\ref{sec:overview}, the $v_{free}$ of step $i$ will be the witness $T_{i+1}$ that ``generates'' $S_{i+1}$ from $S_i$. For example, in Figure~\ref{figure}, $(u,S_1)$ is a red tower on $u$ with witness sequence $(v)$, $(v,S_1,S_2)$ is a blue tower on $v$ with witness sequence $(u,w)$, $(S_1,S_2,S_3)$ is a red tower on $w$ with witness sequence $(w,u)$, and $(S_2,S_3)$ is a blue tower on $u$ with witness sequence $(u)$.

In general, if $(S_{s-1},S_s,\dots,S_f)$ is an $i$-tower on $v$, then it follows from \ref{t2} and \ref{t3} that every vertex in $S_f$ is reachable from $v$ by a path in color $i$ passing through each $S_s, \dots, S_{f-1}$ exactly once. Property \ref{t4}, the set $S_{s-1}$ and the witness sequence $T_s,\dots,T_f$ are needed to establish some pseudorandom properties of the sets $S_s,\dots,S_f$, such as expansion.

To define our auxiliary graph,
we need one additional structure, a \emph{cascade} that we define on a pair of vertices in $L_0$. As
we will see later, it guarantees that the two vertices are connected by
monochromatic paths in a very robust way (Claim \ref{cl:robust}).

\begin{definition}[Cascades]\label{def:cascade}
  Let $i \in [r]$ be a color. We say that two vertices $v,w \in L_0$
  are \emph{connected by an $i$-cascade} if there is an $i$-tower
  $(S^v_{s_v-1},S^v_{s_v}, \dots, S^v_f)$ on $v$ and an $i$-tower
  $(S^w_{s_w-1},S^w_{s_w}, \dots, S^w_f)$ on $w$ for some $1\le s_v,s_w\le f\le
  m$, such that either \begin{enumerate}[label=(C\arabic*)]
  \item \label{c1} $S^v_f = S^w_f$, or
  \item \label{c2} $f = m$ and $e_i(R_v,R_w) \ge e(R_v,R_w)/r$, where $R_v=S^v_m\setminus S^w_m$ and
    $R_w=S^w_m\setminus S^v_m$.
  \end{enumerate}
\end{definition}
Note that there is no need to impose any conditions on the disjointness of the sets in \ref{c2}. This is because if $R_v$ and $R_w$ are small then $S_m^v$ and $S_m^w$ have a significant overlap, so the situation is similar to \ref{c1}.

We now define the auxiliary graph $H$ on the vertex set $L_0$ by adding an
edge of color $i$ between two vertices $v,w\in L_0$ if $v$ and $w$ are connected by
an $i$-cascade. There are two central claims:

\begin{claim}\label{cl:indep}
  W.h.p, $H$ has independence number at most $4r$.
\end{claim}

\begin{claim}\label{cl:haxell}
  W.h.p, for every monochromatic cycle $C$ in $H$ of length at most $K/p$, 
  there is a monochromatic cycle $C^*$ (of the same color) in $G$ such that
  $V(C) \subseteq V(C^*)$.
\end{claim}

From these two claims, Lemma \ref{lemma:tinyset-cover} follows easily:

\begin{proof}[Proof of Lemma \ref{lemma:tinyset-cover}]
  Suppose that $G\sim \Gnp$ is such that the two properties in Claims \ref{cl:indep} and \ref{cl:haxell} hold (this happens w.h.p), so $H$ has independence number at most $4r$.
  Let $Q\subseteq L_0$ be an arbitrary set of size at most $K/p$. As the independence number of $H[Q]$ is also bounded by $4r$, we can apply Theorem \ref{thm:indep} to find a collection of at most $25(\alpha r)^2\log(\alpha r) \leq 400r^4\log (4r^2)$ vertex-disjoint monochromatic cycles in $H[Q]$ covering $Q$.
  
  We know that every cycle in the collection has length at most $|Q|\le K/p$, so we can replace each such cycle $C$ by a monochromatic cycle $C^*$ in $G$ that covers the vertex set of $C$ (using the property in Claim \ref{cl:haxell}). This gives a collection of at most $400r^4\log (4r^2)$ monochromatic cycles in $G$ covering the whole set $Q$.
\end{proof}

\subsection{Proof of Claim \ref{cl:indep}}

We first prove an auxiliary claim:

\begin{claim} \label{claim:cascade_tower}
  W.h.p,
  for every set $\hX=\{x_1, \dots, x_{2r - 1}\}$ of $2r-1$ distinct vertices in
  $L_0$, one of the following two statements holds:
  \begin{enumerate}[label=(\roman*)]
  \item \label{item:casc1} there are two vertices $v,v' \in \hX$ and a color $i
    \in [r]$ such that $v$ and $v'$ are connected by an $i$-cascade, or
  \item \label{item:casc2} there exists a subset $X =\{v_1, \dots, v_r\} \subseteq \hX$ and for each $i \in [r]$ an $i$-tower
      $(S^i_{s_i-1},S^i_{s_i}, \dots, S^i_m)$ on $v_i$, such
      that $S^1_m = \dots = S^r_m$.
  \end{enumerate}
\end{claim}
\begin{proof}
  First, note that w.h.p the property in Lemma \ref{lemma:tuples_expand} holds
  with $\teps=1/4$ simultaneously for every set $L=L_k$ where $k \in [m]$. This
  is because $|L_k| = \Theta(n) \gg \log n /(\teps p^r)$ and $m$ is a constant.
  Thus, we may assume the following property: for every $k\in [m]$ and
  every list of $t\leq \teps/p$ distinct $r$-sets $X_1,\dots,X_t\subseteq [n]\setminus L_k$,
  we have
  \begin{equation}\label{eq:310}
    \Big| \bigcup_{i=1}^t \CN(X_i,L_k) \Big| \geq t|L_k|p^r/2.
  \end{equation}
  This property will be enough to imply Claim \ref{claim:cascade_tower}.

  To prove the claim, assume $\hX = \{x_1,\dots,x_{2r-1}\}$ does not
  satisfy (i), that is, no two vertices in $\hX$ are connected by an $i$-cascade
  for some $i\in [r]$.
  We show that for every $k\in[m]$
  \begin{equation}\label{eq:goal}
    \parbox[c]{10cm}{\centering there is a set $X_k=\{v_1,\dots,v_r\}\subseteq
    \hX$ and an $i$-tower\\ $(S^i_{s_i - 1},S^i_{s_i},\dots,S^i_k)$ on each $v_i$ such
    that $S_k^1=\dots=S_k^r$.}
  \end{equation}
  Note that $X = X_m$ then satisfies \ref{item:casc2}.

  We prove \eqref{eq:goal} by induction on $k$. For the base case $k=1$, let
  $X_1=\{x_1,\dots,x_r\}\subseteq \hX$. Let $Z_1=\CN(X_1,L_1)$
  be the common neighborhood of $X_1$ in $L_1$. By \eqref{eq:310} applied to the
  single $r$-set $X_1$, we have
  \[ |Z_1|= |\CN(X_1,L_1)| \geq |L_1| p^r/2= \mu r^r. \] Let us
  say that a vertex $z \in Z_1$ has the \emph{color pattern}
  $(i_1, \dots, i_r)$ if for each $a \in [r]$, the edge $zx_a$ has
  color $i_a$. There are $r^r$ different color patterns, so we can find a
  subset $S_1 \subseteq Z_1$ of size $|S_1|=\mu$ such that all vertices
  in $S_1$ have the same color pattern $(i_1,\dots,i_r)$.

  Let $S_0^a = \{x_a\}$ for every $a\in [r]$, and $S_1^a = S_1$.
  We now claim that $(S^a_0,S_1)$ is an $i_a$-tower on $x_a$, for every $a\in [r]$.
  For this, choose $T^a_1 = X_1\setminus \{x_a\}$, and let
  us check the conditions \ref{t1}--\ref{t4} separately. 
  First, \ref{t1} requires that $|S_0^a|=1$, $|S^a_1|=\mu$, $S^a_0\subseteq L_0$
  and $S_1\subseteq L_1$, all of which are true.
  For \ref{t2} note that
  $S_1^a \subseteq N_{i_a}(x_a) \cap N(S^a_0)\cap \CN(T^a_1)$ holds because
  every vertex in $S_1^a$ lies in the common
  neighborhood of $X_1=T_1\cup S_0^a$ and has an edge of color $i_a$
  to $x_a$.
  The condition \ref{t3} is vacuous in this case.
  Finally, for \ref{t4}, note that $S_0^a =\{x_a\}$ and $x_a\notin T_1^a$ hold by definition.

  If $i_a=i_b$ for some distinct $a,b\in [r]$, then $x_a$ and $x_b$ are
  connected by an $i_a$-cascade (condition \ref{c1}), contrary to our assumption. So all $r$
  different colors appear in the color pattern $(i_1,\dots,i_r)$, and by
  relabeling the vertices in $X_1=\{x_1,\dots,x_r\}$ as
  $X_1=\{v_1,\dots,v_r\}$ so that $v_{i_a}=x_a$, we get
  the required $i$-tower on each $v_i$, proving \eqref{eq:goal} for
  $k=1$.

  Now suppose that \eqref{eq:goal} holds for some $k-1\ge 1$, and let us prove that it
  holds for $k$, as well. By the induction hypothesis, there is a set of vertices
  $X_{k-1}=\{v_1,\dots,v_r\}\subseteq \hX$ and an $i$-tower
  $(S^i_{s_i-1},S^i_{s_i},\dots,S^i_{k-1})$ on each $v_i$ such that
  $S_{k-1}^1=\dots=S_{k-1}^r=:S_{k-1}$, where by \ref{t1} we have $|S_{k-1}|=\mu^{k-1}$. Let
  $X'= \hX\setminus X_{k-1}=\{w_1, \dots, w_{r-1}\}$ be the remaining $r-1$
  vertices (recall that $\hX$ has size $2r-1$).

  Next, we define $Z_k \subseteq L_k$ as the set of
  all common neighbors of the sets $X'\cup \{v\}$ for $v\in S_{k-1}$, i.e.,
  $$
    Z_k=\bigcup_{v\in S_{k-1}} \CN(X'\cup \{v\},L_k).
  $$  
  Here we have $\mu^{k-1}\le \mu^{m-1} \le \teps/p$ (see \eqref{eq:mu_cow_m}) distinct $r$-sets, so
  \eqref{eq:310} gives
  \[ |Z_k| = \Big|\bigcup_{v\in S'_{k-1}} \CN(X'\cup \{v\},L_k)\Big| \geq
  \mu^{k-1} |L_k| p^r/2 = \mu^{k}r^r. \]
  Each vertex $z \in Z_k$ is connected to $w_j$ for every $j \in [r-1]$
  and has a neighbor in $S_{k-1}$. Fix one such neighbor
  $w'_z\in S_{k-1}$ (chosen arbitrarily) and define the color pattern of $z$ to be $(i_1,\dots,i_r)$
  if the edge $zw_j$ has color $i_j$ for every $j\in[r-1]$ and $zw'_z$ has color
  $i_r$. Once again, there are $r^r$ such patterns, so there is a subset
  $S_k\subseteq Z_k$ of $|S_k|=\mu^k$ vertices that all have the same pattern $(i_1,\dots,i_r)$.

  We first note that for every $j\in [r-1]$, the sequence $(S_{k-1},S_k)$ is an $i_j$-tower on $w_j$
  with the witness sequence $(T_k)$, where $T_k=X'$.
  Verifying the conditions \ref{t1}--\ref{t3} is the same as in the $k=1$ case, and \ref{t4} again holds by definition: $w_j\in T_k$.
  Similarly, $(S^{i_r}_{s_{i_r}-1},S^{i_r}_{s_{i_r}},\dots,S^{i_r}_{k-1},S_k)$ is an
  $i_r$-tower on the vertex $v_{i_r}\in X_{k-1}$. For this, we take a witness sequence
  $T^{i_r}_{s_{i_r}},\dots,T^{i_r}_{k-1}$ of the tower
  $(S^{i_r}_{s_{i_r}-1},S^{i_r}_{s_{i_r}},\dots,S^{i_r}_{k-1})$ and extend it by
  setting
  $T^{i_r}_k=X'$; the conditions \ref{t1}--\ref{t4} are then easy to check
  using the fact that $(S^{i_r}_{s_{i_r}-1},S^{i_r}_{s_{i_r}},\dots,S^{i_r}_{k-1})$ is
  already an $i_r$-tower on $v_{i_r}$: for \ref{t1}, we already have $|S_k|=\mu^k$, whereas for \ref{t3}, we know that every $z\in S_k$ is in the common neighborhood of $X'$, and is also connected to $S_{k-1}$ by an $i_r$-colored edge, so $S_k\subseteq N_{i_r}(S^{i_r}_{k-1})\cap N^*(T^{i_r}_k)$. Every other requirement holds by induction.

  If $i_a=i_b$ for some distinct $a,b\in[r]$, then two vertices in
  $X_k=\{w_1,\dots,w_{r-1},v_{i_r}\}\subseteq \hX$ are connected by an
  $i_a$-cascade (again by condition \ref{c1}), which we had ruled out. So all $r$ colors appear in
  $(i_1,\dots,i_r)$ exactly once, and thus we can relabel the vertices in $X_k$
  as $v_1,\dots,v_r$ to get the desired $i$-towers on the (new) $v_i$'s.
\end{proof}

\begin{proof}[Proof of Claim \ref{cl:indep}]
  Recall that $H$ is the graph on vertex set $L_0$ where we add an edge $vw$ in color $i$ whenever $v$ and $w$ are connected by an $i$-cascade in $G$.
  
  Assume that $G$ satisfies the property in Claim
  \ref{claim:cascade_tower}, as it does w.h.p, and
  let $X\subseteq L_0$ be a set of $4r-2$ vertices.
  We show that $X$ contains two distinct vertices connected by
  an $i$-cascade, for some $i\in[r]$.
  For this, split $X$ into two
  disjoint sets $\hX,\hX'$ of $2r-1$ vertices each. If either set contains two
  vertices connected by an $i$-cascade for some $i\in[r]$, then we are done.
  Otherwise, by Claim \ref{claim:cascade_tower}, there are subsets $\{v_1,
  \dots, v_r\} \subseteq \hX$ and $\{v'_1, \dots, v'_r\} \subseteq \hX'$ and
  an $i$-tower $(S^i_{s_i-1},S^i_{s_i}, \dots, S^i_m)$ on $v_i$ and
  another $i$-tower $(S'^i_{s'_i-1},S'^i_{s'_i}, \dots, S'^i_m)$ on $v'_i$ for
  every color $i\in [r]$, such that $S^1_m = \dots = S^r_m=:S_m$ and $S'^1_m =
  \dots = S'^r_m=:S'_m$. Let $i \in [r]$ be a color such that $e_i(R,R')
  \ge e(R,R')/r$, where $R=S_m\setminus S_m'$ and $R'=S_m'\setminus S_m$. Then
  $v=v_i$ and $v'=v'_i$ are connected by an $i$-cascade, and we
  are again done. 
\end{proof}

\subsection{Proof of Claim \ref{cl:haxell}}

Let us recall the statement: given, say, a red cycle $C$ of size $O(1/p)$ in the auxiliary graph $H$, we want to find a red cycle $C^*$ in $G$ that covers all the vertices of $C$. We use the following strategy.

For each edge $e$ in $C$, we have a cascade on top of it, so we know that there are many short (of length at most $2m + 1$) red paths connecting the endpoints of $e$, all with internal vertices in $U$. Let us denote the set of internal vertices of these paths by $\mathcal{H}_e = \{P_e^1, \dots, P_e^{\ell}\}$. In order to create a cycle $C^*$, it suffices to choose exactly one $P_e \in \mathcal{H}_e$ for each $e$ so that they are all pairwise disjoint. We use Theorem \ref{thm:hyp_matching} to achieve this. Taking the $\calH_e$ as the $(2m)$-uniform hypergraphs, it is enough to check that for any subset $E' \subset E(C)$ and any $Y \subseteq U$ of size $|Y| \le 4m|E'|$ there is a path $P_e \in \calH_e$, for some $e \in E'$ that completely avoids $Y$. 

The proof relies on two ingredients. First, we will show that most cascades (or, rather, the associated towers) corresponding to the edges in $E'$ are disjoint on almost all levels (Claim \ref{claim:towers_disjoint}), thus it is impossible for $Y$ to significantly intersect each of them. Second, if we remove a small fraction of the vertices from each level of a cascade connecting $v$ and $w$, then it still contains a path from $v$ to $w$ (Claim \ref{cl:robust}).

The following notion of \emph{independent towers} will be crucial for our applications of Lemma \ref{lemma:tuples_expand}.

\begin{definition}[Independent towers]
  Suppose that $1\leq s\leq f\leq m$ and that we are given
  a collection $\{(S_{s-1}^j,S_s^j,\dots,S_f^j) \}_{j\in [t]}$ of $t$
  towers\footnote{More
  precisely, each $(S_{s-1}^j,S_s^j,\dots,S_f^j)$ is an $i_j$-tower on some vertex
  $v_j\in L_0$, but the precise values for $i_j$ and $v_j$ do not matter
  here. We use this kind of sloppiness throughout the proof.}.
  We say that this collection is \emph{independent} if there exists
  a witness sequence $(T_s^j,\dots,T_f^j)$ for every tower
  $(S_{s-1}^j,\dots,S_f^j)$ such that all sets of the form $T_s^j \cup \{v\}$, 
  where $j\in [t]$ and $v \in S_{s-1}^j$, are \emph{distinct}.
\end{definition}

Our next claim says that among towers on distinct vertices,
there is always a large subset of independent towers.

\begin{claim}\label{claim:indep}
  Let $1\leq s\leq f\leq m$ and let $t\geq 0$. Let $v_1,\dots,v_t$ be distinct
  vertices in $L_0$ and for each $j\in [t]$, let $\calT_j =
  (S_{s-1}^j,\dots,S_f^j)$ be a tower on $v_j$. Then there is a set
  $\calI\subseteq [t]$ of at least $t/r$ indices such that the towers
  $\{\calT_j\}_{j\in \calI}$ are independent.
\end{claim}

\begin{proof}
  We define a graph $G_T$ on the vertex set $\{\calT_1,\dots,\calT_t\}$
  where $\calT_j$ and $\calT_{j'}$ for $j\neq j'$ are connected by an edge if
  and only if $T^j_s\cup \{v\} = T^{j'}_s\cup \{v'\}$ for some $v\in S_{s-1}^j$
  and $v'\in S_{s-1}^{j'}$. We will show that the maximum degree of $G_T$ is 
  at most $r-1$. This clearly implies that $G_T$ contains an independent set of size 
  at least $t/r$ (e.g. by choosing its vertices greedily), which is exactly what we need.

  We first show that if $\calT_j$ and $\calT_{j'}$ are adjacent in $G_T$ then $v_{j'}\in T_s^j$. 
  Indeed, if $s=1$, then for each $j\in [t]$ we have $S_0^j = \{v_j\}$ by \ref{t4}. Hence 
  $\calT_j$ and $\calT_{j'}$ can only be adjacent if $T_1^j\cup \{v_{j'}\}=T_1^{j'}\cup\{v_j\}$,
  but then $v_j\ne v_{j'}$ implies $v_{j'}\in T_1^j$. On the other hand, if $s>1$ then 
  $T_s^j\subseteq L_0$ and $S_{s-1}^j\subseteq L_{s-1}$ where $L_0\cap L_{s-1}=\emptyset$, 
  so $\calT_j$ and $\calT_{j'}$ can only be adjacent if $T_s^j=T_s^{j'}$. Once again, \ref{t4} 
  then implies that $v_{j'}\in T_s^{j'}=T_s^j$.
  But $T_s^j$ has $r-1$ elements, so there are at most $r-1$ different choices for this $v_{j'}$.
  Thus $\calT_j$ has at most $r-1$ neighbors in $G_T$.
\end{proof}

The next claim shows that in every small collection of independent towers, a significant fraction of them are almost mutually disjoint on any given level $k$.

\begin{claim} \label{claim:towers_disjoint}
  The following holds with high probability. For every family $\{(S_{s-1}^j, \dots, S_{f}^j)\}_{j \in [t]}$ of $t$ independent $i$-towers, where $1 \le s \le f \le m$ and $t \le n^{\eps}/ (\mu^k p)$ for some $k \in \{s, \dots, f\}$, there is a set $\calI_k \subseteq [t]$ of size $|\calI_k| \ge t / 2^k$ such that 
  $$
    \Big| S_k^j \cap \bigcup_{j' \in \calI_k \setminus \{j\}} S_k^{j'} \Big| \le (80 r^r)^k \cdot n^{-\eps/4} \mu^k
  $$
  for every $j \in \calI_k$.
\end{claim}

\begin{proof}
  Note that with high probability Lemma \ref{lemma:tuples_expand} applies with $\teps=n^{-\eps/2}$ and
  $L=L_k$ simultaneously for every $k\in[m]$. Thus, we may assume that for
  every $k\in [m]$ and every family of $T \leq \teps/p$ different $r$-sets
  $X_1,\dots,X_T\subseteq [n]\setminus L_k$, we have
  \begin{equation}\label{eq:310a}
    \Big|\bigcup_{\ell=1}^T\CN(X_\ell,L_{k})\Big| \in (1 \pm \sqrt{\teps})T|L_{k}|p^r = (1 \pm \sqrt{\teps}) T \cdot 2r^r \mu.
  \end{equation}
  We derive Claim \ref{claim:towers_disjoint} from this property. For convenience, let $\xi_k = (80 r^r)^k n^{-\eps/4}$ for $k \in [m]$.

  Let $\{(S^j_{s-1},S_s^j,\dots,S^j_f)\}_{j\in [t]}$ be a collection of $t$
  independent $i$-towers as in the statement of the claim. For each $j\in [t]$,
  let $(T_s^j,\dots,T_f^j)$ be a corresponding witness sequence such
  that all sets of the form
  $T_s^j\cup \{v\}$, where $j\in [t]$ and $v\in S_{s-1}^j$, are distinct (this
  is possible because of the independence). We prove the claim by induction on $k$.

  Suppose $k = s$. For each $j \in [t]$ define $Z^j_k = N(S_{k-1}^j) \cap N^*(T_k^j, L_k)$, and set 
  $$
    Z_k = \bigcup_{j \in [t]} Z_k^j \qquad \text{ and } \qquad \hat Z_k = \bigcup_{j \neq j' \in [t]} Z_k^j \cap Z_k^{j'}.
  $$
  Note that the definition $Z_k^j$ is almost the same as the definition of $S_k^j$, however without restricting the neighborhood of $S_{k-1}^j$ to a specific color. Therefore, we have $S_k^j \subseteq Z_k^j$. Alternatively, we can define $Z_k^j$ as
  $$
    Z^j_k=\bigcup_{v\in S_{k-1}^j}\CN(T^j_k\cup \{v\},L_k). 
  $$
  By \eqref{eq:310a}, each $Z_k^j$ is of size $|Z_k^j| \in (1 \pm \sqrt{\teps}) 2r^r \mu^k$ (recall that $|S_{k-1}^j| = \mu^{k-1}$). We prove a somewhat stronger statement than needed, namely that the set $\calI_k \subseteq [t]$ consisting of all indices $j \in [t]$ such that
  $$
    \Big| Z_k^j \cap \hat Z_k \Big| = \Big| Z_k^j \cap  \bigcup_{j' \in [t] \setminus \{j\}}  Z_k^{j'}   \Big| \le 10 \sqrt{\teps} \cdot 2r^r \mu^k < \xi_k \mu^k,
  $$
  is of size $|\calI_k| \ge t / 2$. 

  To this end, we first estimate the size of $Z_k$ and $\hat Z_k$. As the towers are independent, there are exactly $\sum_{j \in [t]} |S_{k-1}^j| = t \mu^{k-1} < 1 / (n^{\eps/2} p)= \teps/p$ sets of the form $T_k^j \cup \{v\}$ where $j \in [t]$ and $v \in S_{k-1}^j$ (recall $\mu = \Theta(n^{r \eps})$). Thus applying \eqref{eq:310a} again we obtain $|Z_k| \in (1 \pm \sqrt{\teps}) t \cdot 2r^r \mu^k$, 
  which in turn gives the following estimate on the size of $\hat Z_k$.
  $$
    |\hat Z_k| \le \Big( \sum_{j \in [t]} |Z_k^j|  \Big) - | Z_k | \le 2 \sqrt{\teps} t\cdot 2r^r \mu^k.
  $$ 
  Indeed, here $|\hat Z_k|$ is the number of elements in $L_k$ that are counted at least twice by the sum $\sum_{j \in [t]} |Z_k^j|$, whereas $|Z_k|$ is the number of elements counted at least once.
  
  Putting everything together, we have the following:
  \begin{align*}
    (1 - \sqrt{\teps}) t \cdot 2r^r \mu^k &\le | Z_k | = \sum_{j \in [t]} |Z_k^j \setminus \hat Z_k| + |\hat Z_k| \\
    &\le \sum_{j \in [t]} |Z_k^j| - \sum_{j \in [t] \setminus \mathcal{I}_k} |Z_k^j \cap \hat Z_k| + |\hat Z_k| \\
    &\le (1 + \sqrt{\teps}) t  \cdot 2r^r \mu^k - (t - |\mathcal{I}_k|) 10 \sqrt{\teps} \cdot 2r^r \mu^k + 2 \sqrt{\teps} t \cdot 2r^r \mu^k\\
    &\le (1 + 3 \sqrt{\teps} - 5 (t - |\mathcal{I}_k|) \tfrac{2\sqrt{\teps}}{t}) t \cdot 2r^r \mu^k.
  \end{align*}
  This implies $|\mathcal{I}_k| > t/2$, as required.
  
  \medskip
  Next, suppose that $k > s$ and the claim holds for $k - 1$. As $t \le n^{\eps}/ (\mu^k p) \le n^{\eps}/ (\mu^{k-1} p)$, we can apply the induction hypothesis to obtain a family $\calI_{k-1} \subseteq [t]$ of $t'=|\calI_{k-1}|\ge t/2^{k-1}$ almost disjoint towers on the $(k-1)$'st level. For each $j \in \calI_{k-1}$ let
  $$
    \hS_{k-1}^{j} = S_{k-1}^{j} \setminus \bigcup_{j' \in \calI_{k-1} \setminus \{j\}} S_{k-1}^{j'}.
  $$
  Then these sets are all disjoint and have size
  \begin{equation} \label{eq:hS}
    (1 - \xi_{k-1}) \mu^{k-1} \le |\hS_{k-1}^j| \le \mu^{k-1}.
  \end{equation} 
  Note that $\hS_{k-1}^j\subseteq L_{k-1}$ with $k>1$, so $\hS_{k-1}^j$ is also disjoint from $T_k^j\subseteq L_0$. More importantly, these facts imply that all the sets of the form $T_k^j \cup \{v\}$ for $j \in \calI_{k-1}$ and $v \in \hS_{k-1}^j$ are distinct. Now we can argue as in the base case.

  For each $j \in \calI_{k-1}$ define 
  $$
    Z_k^j = N(\hS_{k-1}^{j}) \cap N^*(T_k^j, L_k) = \bigcup_{v \in \hS_{k-1}^j} N^*(T_k^j \cup \{v\}, L_k)
  $$
  and let
  $$
    Z_k = \bigcup_{j \in \calI_{k-1}} Z_k^j \qquad \text{ and } \qquad \hat Z_k = \bigcup_{j \ne j' \in \calI_{k-1}} Z_k^j \cap Z_k^{j'}.
  $$
  Observe that, unlike in the base case, we do not have $S_k^j \subseteq Z_k^j$: some vertices in $S_k^j$ might only have neighbors in $S_{k-1}^j \setminus \hS_{k-1}^j$. However, $S_{k-1}^j \setminus \hS_{k-1}^j$ is small, so $S_k^j \setminus Z_k^j$ will turn out to be negligible. We will deal with these vertices at the end. 

  From $|\hS_{k-1}^j| \le \mu^{k-1}\le \teps/p$ and \eqref{eq:310a} we have 
  $$
    |Z_k^j| \le (1 + \sqrt{\teps}) \cdot 2r^r \mu^k \le (1 + \xi_{k-1}) \cdot 2r^r \mu^k.
  $$
  As already mentioned, sets of the form $T_k^j \cup \{v\}$ for $j \in \calI_{k-1}$ and $v \in \hS_{k-1}^j$ are all distinct, thus by \eqref{eq:hS}, there are 
  $$
    \sum_{j \in \calI_{k-1}} |\hS_{k-1}^j| \ge (1 - \xi_{k-1}) \mu^{k-1} t' 
  $$
  many of them. Therefore, applying \eqref{eq:310a} again we get $|Z_k| \ge (1 - 2 \xi_{k-1}) t' \cdot 2r^r \mu^k$ which, in turn, gives the following bound on the size of $\hat Z_k$.
  $$
    |\hat Z_k| \le \Big( \sum_{j \in \calI_{k-1}} |Z_k^j| \Big) - |Z_k| \le 3 \xi_{k-1} t' \cdot 2r^r \mu^k.
  $$
  We now define $\calI_k \subseteq \calI_{k-1}$ as the set of all indices $j \in \calI_{k-1}$ such that
  $$
    \Big| Z_k^j \cap \hat Z_k \Big| < 20 \xi_{k-1} \cdot 2r^r \mu^k = \tfrac{\xi_{k}}{2} \mu^k.
  $$
  Using the above estimates on the size of $Z_k$ and $\hat Z_k$, we get the following:
  \begin{align*}
    (1 - 2 \xi_{k-1}) t' \cdot 2r^r \mu^k &\le | Z_k | = \sum_{j \in \calI_{k-1}} |Z_k^j \setminus \hat Z_k| + |\hat Z_k| \\
    &\le \sum_{j \in \calI_{k-1}} |Z_k^j| - \sum_{j \in \calI_{k-1} \setminus \mathcal{I}_k} |Z_k^j \cap \hat Z_k| + |\hat Z_k| \\
    &\le (1 + \xi_{k-1}) t' \cdot 2r^r \mu^k - (t' - |\mathcal{I}_k|)  20 \xi_{k-1} \cdot 2r^r \mu^k + 3 \xi_{k-1} t' \cdot 2r^r \mu^k\\
    &\le (1 + 4 \xi_{k-1} - (t' - |\mathcal{I}_k|) 20 \tfrac{\xi_{k-1}}{t'}) t' \cdot 2r^r \mu^k,
  \end{align*}
  which implies $|\calI_k| > t' / 2 \ge t/2^k$.

  So far we have shown that $\calI_k$ has the desired size, and that for each $j \in \calI_k$, the intersection of $S_k^j$ with other sets $S_k^{j'}$ \emph{inside} $Z_k^j$ contains at most $\xi_k \mu^k / 2$ vertices. Therefore, it suffices to prove $|S_k^j \setminus Z_k^j| \le \tfrac{\xi_k}{2} \mu^k$ to finish the proof. Note that 
  $$
    S_k^j \setminus Z_k^j \subseteq N(S_{k-1}^j \setminus \hS_{k-1}^j) \cap N^*(T_k^j, L_k) = \bigcup_{v \in S_{k-1}^j \setminus \hS_{k-1}^j} N^*(T_k^j \cup \{v\}, L_k).
  $$
  As $|S_{k-1}^j \setminus \hS_{k-1}^j| \le \xi_{k-1} \mu^{k-1}$, from \eqref{eq:310a} we get
  $$
    |S_k^j \setminus Z_k^j| \le (1 + \sqrt{\teps}) \xi_{k-1} \cdot 2r^r \mu^k < \xi_k \mu^k / 2,
  $$
  as required. This concludes the argument.
\end{proof}

Next, we show that cascades are resilient to small changes.

\begin{claim} \label{cl:robust}
    There is a constant $c > 0$ such that the following holds with high probability. Suppose $v, w \in L_0$ are connected by an $i$-cascade with underlying $i$-towers $\calT_v = (S_{s_v-1}^v, \dots, S_f^v)$ and $\calT_w = (S^w_{s_w-1}, \dots, S_f^w)$, for some $1 \le s_v, s_w \le f \le m$. Then for any subset $Y \subseteq U$ such that $|S_k^u \cap Y| \le c \mu^k$ for every $u \in \{v, w\}$ and $k \in \{s_u, \dots, f\}$, the graph $G[(\{v,w\} \cup U) \setminus Y]$ contains an $i$-colored $v$-$w$ path of length at most $2f + 1$.
\end{claim}
\begin{proof}
  With high probability, Lemma \ref{lemma:tuples_expand} applies with
  with $\teps=1/4$ and $L = L_k$ simultaneously for every $k \in [m]$. Thus, we may assume that for every $k \in [m]$ and every family of $t \le \teps / p$ different $r$-sets $X_1, \dots, X_t \subseteq [n] \setminus L_k$ we have
  \begin{equation} \label{eq:expand_robust}
    \Big| \bigcup_{j \in [t]} N^*(X_j, L_k) \Big| \le 2 t |L_k|p^r = 4 r^r t \mu.
  \end{equation}
  Furthermore, by Lemma \ref{lemma:densityXY}, we can assume that for every disjoint $X, X' \subseteq V(G)$ of size $|X|, |X'| \gg \log n / p$ we have
  \begin{equation} \label{eq:dens}
    e(X, X') \in (1 \pm \alpha) |X||X'| p,
  \end{equation}
  for $\alpha = 1 / (8r+8)$. We show that these properties suffice to derive the claim.

  Consider some $u \in \{v, w\}$. We first show that there is a subset $B^u \subseteq S_f^u$ of size at most $\alpha |S_f^u| = \alpha \mu^f$, such that for every vertex $u' \in S_f^{u} \setminus B^u$, there is an $i$-colored $u$-$u'$ path avoiding $Y$ of length at most $f$. 

  To this end, we define the sets $B^u_{s_u}, \dots, B^u_f$ level by level as follows. Let $B_{s_u}^u = Y \cap S_{s_u}^u$ and then iteratively set $B_k = N(B_{k-1}, S_k^u) \cup (Y \cap S_k^u)$ for $k \in \{s_u + 1, \dots, f\}$. It is easy to see by induction on $k\in\{s_u,\dots,f\}$ (and using the definition of an $i$-tower) that for every vertex 
  $u' \in S_k^u \setminus B_k^u$, the graph $G[(\{v\} \cup U) \setminus Y]$ contains an $i$-colored $u$-$u'$ path of length at most $k$. We will prove, using induction on $k$, that
  \begin{equation}\label{eq:blub}
    |B_k^u|\le (8r^r)^{k} c \mu^k \quad\text{for every $k\in \{s_u,\dots,f\}$}.
  \end{equation}
  Choosing a $c<\alpha/(8r^r)^f$ in the assumptions of the claim then ensures that $B^u = B^u_f$ is of size at most $\alpha \mu^f$.

  The case $k=s_u$ follows from the assumptions on $Y$ and the definition of $B_{s_u}^u$, so let $k>s_u$ and assume that \eqref{eq:blub} holds for $k-1$. By the definition of $S_k^u$, we have
  \[
    N(B_{k-1}^u, S_k^u) \subseteq N(B_{k-1}^u) \cap N^*(T_k^u,L_k) \subseteq \bigcup_{v \in B_{k-1}^u} N^*(T_k^u \cup \{v\}, L_k).
  \]
  As $B_{k-1}^u$ is asymptotically smaller than $1/p$ (see \eqref{eq:mu_cow_m}), we can use \eqref{eq:expand_robust} to get
  \[
    |N(B_{k-1}^u, S_k^u)| \le 4 r^r |B_{k-1}^u|  \mu \le 4 r^r (8r^r)^{k-1} c \mu^k. 
  \]
  The assumption of the claim states $|S_k^u \cap Y| \le c \mu^k$, which implies the desired bound on $B_k^u$ (with room to spare).

  We now use these sets to find a desired path from $v$ to $w$. The towers $\calT_{v}$ and $\calT_{w}$ form an $i$-cascade connecting
  $v$ and $w$, thus by definition, we either have $S^{v}_f=S^{w}_f$, or
  $f = m$ and $e_i(R_v,R_w)\ge e(R_v,R_w)/r$ where $R_v= S^{v}_m\setminus S^{w}_m$ and $R_w= S^{w}_m\setminus S^{v}_m$. In the former case we are immediately done: since
  $|B^{v} \cup B^{w}| < 2 \alpha |S^{v}_f|$ there is a vertex $z \in S^{v}_f \setminus (B^{v} \cup B^{w})$ and hence an $i$-colored $v$-$z$-$w$ walk, containing an $i$-colored path,
  of length at most $2f \le 2m+1$, disjoint from $Y$. Similarly, if we are in the latter case and $|S^{v}_m \cap S^{w}_m| > 2 \alpha \mu^m$, then we are done for the same reason.
  
  Let us therefore assume that we are in the latter case and
  $|R_v|, |R_w| \ge (1 - 2\alpha) \mu^m$. It is enough to show that there is an edge $z_v z_w \in G$ of color $i$ such that $z_v\in R_v\setminus B^{v}$ and
  $z_w\in R_w\setminus B^{w}$. Indeed, such an edge would connect an $i$-colored $v$-$z_v$ path and an $i$-colored $w$-$z_w$ path that both avoid $Y$, thus providing a desired path from $v$ to $w$ of length at most $2m + 1$. To show that such an edge exists, note that $|R_v|,|R_w|\gg (\log n)/p$ (see \eqref{eq:mu_cow_m}), so \eqref{eq:dens} gives
  \[
    e_i(R_v,R_w) \ge e(R_v,R_w)/r \ge |R_v||R_w|p/2r \ge
    (1-4\alpha)|S^{v}_m||S^{w}_m|p/r.
  \]
  On the other hand, as $B^{v}$ is contained in some set $B$ of size $\alpha |S^v_m|$, it touches at most 
  \[
    e(B^v, R_w) \le e(B,R_w) \le 2|B||R_w|p\leq 2\alpha |S_m^{v}| |S_m^{w}|p.
  \] 
  of these edges (again, using \eqref{eq:dens}). We similarly get that $B^{w}$ touches at most $2\alpha |S_m^{v}||S_m^{w}|p$ such edges, which means that
  \[
  e_i(R_v\setminus B^{v}, R_w\setminus B^{w}) \ge (1-(4+4r)\alpha)|S^{v}_m||S^{w}_m|p/r >0
  \] 
  with $\alpha=1/(8r+8)$, so at least one edge avoids both $B^v$ and $B^w$.
\end{proof}

We finally have all the necessary tools to prove Claim \ref{cl:haxell}.

\begin{proof}[Proof of Claim \ref{cl:haxell}]
  Suppose that $G$ satisfies the statement in Claim \ref{claim:towers_disjoint} and Claim \ref{cl:robust}, as it does w.h.p, and let $c > 0$ be a constant given by Claim \ref{cl:robust}. We show that these two properties are enough to prove Claim \ref{cl:haxell}.

  Let $C$ be an $i$-colored cycle in $H$ for some color $i \in [r]$.
  For every edge $e=vv'$ of $C$, we define a $2m$-uniform hypergraph
  $\calH_e=(U,\calE_e)$ on the vertex set $U$, where a
  set $A\subseteq U$ of size $2m$ belongs to $\calE_e$ if $G[A\cup \{v,v'\}]$
  contains an $i$-colored $v$-$v'$ path. Note that if for every edge $e\in
  E(C)$ we can find a hyperedge $f_e \in \calE_e$ such that all these hyperedges
   are pairwise disjoint, then the corresponding paths
  form an $i$-colored cycle $C^*$ in $G$ such that $V(C^*) \cap L_0 = V(C) \cap L_0$. It is therefore enough to show that such a family of hyperedges exists for every monochromatic cycle
  $C$ in $H$ of length at most $K/p$.

  This will follow from Theorem \ref{thm:hyp_matching}, provided that for every $E'\subseteq E(C)$, we have
  \begin{equation}\label{eq:hax}
    \tau\big(\bigcup_{e\in E'}\calE_e\big)>(4m-1)(|E'|-1),
  \end{equation}
  where $\tau(\calE)$ is the smallest size of a set $X\subseteq U$ intersecting
  every hyperedge in $\calE$.

  Consider some $E'\subseteq E(C)$, and label the endpoints of each edge $e\in E'$ with $v_e$ and $w_e$. We first pass to a large subset $F \subseteq E'$ that has some convenient properties. This is done in three steps.  

  First, let $E''\subseteq E'$ be a subset of size $|E'|/3$ such that the edges in $E''$ are pairwise disjoint (i.e., no two edges in $E''$ share an endpoint). This is possible because all the edges of $E'$ lie on a cycle.
  Next, recall from the definition of $H$ that each edge $e\in E''$ represents
  an $i$-cascade between $v_e$ and $w_e$ formed by an $i$-tower
  $\calT_{v_e}=(S^{v_e}_{s_{v_e}-1},\dots, S^{v_e}_{f_e})$ on $v_e$ and another
  $i$-tower $\calT_{w_e}=(S^{w_e}_{s_{w_e}-1},\dots, S^{w_e}_{f_e})$ on $w_e$
  (see Definition \ref{def:cascade}). Note that $f_e$ is the same for both
  vertices $v_e$ and $w_e$.
  As there are $m$ possible values for each of $s_v,s_w$ and $f_e$, we can
  find a subset $E''' \subseteq
  E''$ of size at least $|E''|/m^3 \geq |E'|/(3m^3)$ and levels
  $s_v,s_w,f\in[m]$ such that
  $s_{v_e}=s_v$,
  $s_{w_e}=s_w$ and $f_e=f$ for every $e\in E'''$. Finally, applying Claim \ref{claim:indep} twice (once for the
  the collection $\{\calT_{v_e}\}_{e\in E'''}$ and once for the collection $\{\calT_{w_e}\}_{e\in E'''}$), we find a subset
  $F \subseteq E'''$ of size $|E'|/(3m^3r^2)$ such that $\{\calT_{v_e}\}_{e\in F}$ and
  $\{\calT_{w_e}\}_{e\in F}$ are both independent collections of towers (but their union might not be).

  Let $t = |F| \leq K/p$ and let $e_1,\dots,e_t$ be the edges in $F$. Rephrasing \eqref{eq:hax}, we need to prove that no set $Y \subseteq U$ of size $|Y| = 4m|E'|
  \le 12m^4r^2 t$ covers all the hyperedges
  in $\bigcup_{e \in F} \calE_e$. That is, there is an edge $e=v_ew_e$ in $F$ and
  an $i$-colored path $P$ of length at most $2m+1$ connecting $v_e$ and $w_e$
  such that $V(P)\subseteq (\{v_e, w_e\} \cup U) \setminus Y$. Claim \ref{cl:robust} suggests that it suffices to show that there is an edge $e \in F$ whose cascade mostly evades $Y$.

Let $Y^{v_e}_k=S^{v_e}_k \cap Y$ for every $e \in F$ and $k\in \{s_v,\dots, f\}$. We will show that for each $k \in \{s_v, \dots, f\}$ most sets $Y^{v_e}_k$ are quite small. More precisely, the set $\calB_k \subseteq F$ of all edges $e \in F$ such
that $|Y_{k}^{v_e}|\geq c \mu^k$ is of size
\begin{equation}\label{eq:bla}
  |\calB_k|< t/(2m).
\end{equation}
Consider some $k \in \{s_v, \dots, f\}$. First, we show that $|\calB_k|\le n^{\eps}/(\mu^k p)$. Indeed, if this is not the case then choose an arbitrary subset 
$\calJ\subseteq \calB_k$ of size $n^{\eps}/(\mu^k p)$ and let $\calI_k \subseteq \calJ$ be the subset provided by Claim \ref{claim:towers_disjoint} when applied to the towers $\{ \calT_{v_e} \}_{e \in \calJ}$. For each $e \in \calI_k$, the set $S_k^{v_e}$ intersects $\bigcup_{e' \in \calI_k} S_k^{v_{e'}}$ on $o(\mu^k)$ vertices, thus it contains at least $c \mu^k / 2$ unique elements from $Y$ (i.e. elements which do not appear in any other $S_k^{v_e'}$ for $e' \in \calI_k$). This implies
$$
  |Y| \ge \tfrac{c}{2} \mu^k |\calI_k| \ge \tfrac{c}{2} \mu^k \cdot n^{\eps} / (2^k \mu^k p) \gg t,
$$
which is a contradiction. Therefore $|\calB_k| \le n^{\eps} / (\mu^k p)$, thus we can apply Claim \ref{claim:towers_disjoint} to all the towers $\{ \calT_{v_e} \}_{e \in \calB_k}$. Let $\calI_k$ be the obtained set of indices and, again, note that each $S_k^{v_e}$ contains at least $c \mu^k / 2$ unique elements from $Y$. Assuming $|\calB_k| \ge t / (2m)$, the contradiction follows similarly as in the previous case:
$$
  |Y| \ge \tfrac{c}{2} \mu^k |\calI_k| \ge \tfrac{c}{2} \mu^k \cdot t / (2^k \cdot 2m) \gg t.
$$

Finally, it follows from \eqref{eq:bla} that there is a subset $\calI_v \subseteq F$ of size $|\calI_v|>t/2$ such that for all $e \in \calI_v$ and all $k\in\{s_v,\dots, f\}$, we have $|Y_k^{v_e}|\le c \mu^k$.
By the same argument, a subset $\calI_w \subseteq F$ of size $|\calI_w|>t/2$ exists for the other endpoints of the edges, as well, such that $|Y_k^{w_e}|\le c \mu^k$ for every $e \in \calI_w$ and all $k\in\{s_w,\dots, f\}$. In particular, there is an edge $e \in \calI_v\cap \calI_w$ such that $Y$ intersects at most a $c$-fraction of each level of the cascade on $e$. The existence of a desired path now follows from Claim \ref{cl:robust}. 
\end{proof}

\section{Concluding remarks}

In this paper we made a step towards the random analog of the theorem of
Erd\H{o}s, Gy\'arf\'as and Pyber \cite{erdos1991vertex} on monochromatic cycle covers. Our result
leaves a few interesting open problems:
\begin{itemize}  
  \item The most interesting open problem is to show that there is
    a \emph{partition} of the vertices of $\Gnp$ into constantly many
    monochromatic cycles (or paths), even for some larger
    values of $p$.

  \item It would be nice to give a more precise estimate on the threshold for the
    property that every $r$-coloring of the edges of $\Gnp$ admits
    a vertex cover by a number of monochromatic cycles depending only on $r$.
    In view of the construction of Bal and DeBiasio that we mentioned in the introduction,
    it seems natural to guess that the threshold should be of
    the order $(\log n / n)^{1/r}$. Note that in the proof of Lemma
    \ref{lemma:tinyset-cover} we heavily rely on the fact that there are only
    constantly many levels in a cascade, which requires 
    $p\geq n^{-1/r+\eps}$ for some constant $\eps>0$.

  \item We did not put much effort into optimizing the number of cycles we cover with 
    in Theorem~\ref{thm:mainr2}, thus it could most likely
    be improved. It would be interesting to see if one could obtain similar bounds
    to the case $G = K_n$, e.g., do $O(r\log r)$ cycles suffice?

\end{itemize}

\bibliographystyle{abbrv}
\bibliography{references}

\begin{thebibliography}{10}

\bibitem{bal2015partitioning}
D.~Bal and L.~DeBiasio.
\newblock Partitioning random graphs into monochromatic components.
\newblock {\em The Electronic Journal of Combinatorics}, 24(1):P1--18, 2017.

\bibitem{balogh2014partitioning}
J.~Balogh, J.~Bar{\'a}t, D.~Gerbner, A.~Gy{\'a}rf{\'a}s, and G.~N.
  S{\'a}rk{\"o}zy.
\newblock Partitioning 2-edge-colored graphs by monochromatic paths and cycles.
\newblock {\em Combinatorica}, 34(5):507--526, 2014.

\bibitem{bessy2010partitioning}
S.~Bessy and S.~Thomass{\'e}.
\newblock Partitioning a graph into a cycle and an anticycle, a proof of
  {L}ehel's conjecture.
\newblock {\em Journal of Combinatorial Theory, Series B}, 100(2):176--180,
  2010.

\bibitem{conlon2014combinatorial}
D.~Conlon.
\newblock Combinatorial theorems relative to a random set.
\newblock In {\em Proceedings of the International Congress of Mathematicians
  2014}, volume~4, pages 303--328. 2014.

\bibitem{debiasio2017monochromatic}
L.~DeBiasio and L.~L. Nelsen.
\newblock Monochromatic cycle partitions of graphs with large minimum degree.
\newblock {\em Journal of Combinatorial Theory, Series B}, 122:634--667, 2017.

\bibitem{erdos1991vertex}
P.~Erd{\H{o}}s, A.~Gy{\'a}rf{\'a}s, and L.~Pyber.
\newblock Vertex coverings by monochromatic cycles and trees.
\newblock {\em Journal of Combinatorial Theory, Series B}, 51(1):90--95, 1991.

\bibitem{gyarfas67}
L.~Gerencs\'er and A.~Gy\'arf\'as.
\newblock On {R}amsey-type problems.
\newblock {\em Ann. Univ. Sci. Budapest. E\"otv\"os Sect. Math.}, 10:167--170,
  1967.

\bibitem{grinshpun2016monochromatic}
A.~Grinshpun and G.~N. S{\'a}rk{\"o}zy.
\newblock Monochromatic bounded degree subgraph partitions.
\newblock {\em Discrete Mathematics}, 339(1):46--53, 2016.

\bibitem{gyarfas1989covering}
A.~Gy{\'a}rf{\'a}s.
\newblock Covering complete graphs by monochromatic paths.
\newblock In {\em Irregularities of partitions}, pages 89--91. Springer, 1989.

\bibitem{gyarfas2016vertex}
A.~Gy{\'a}rf{\'a}s.
\newblock Vertex covers by monochromatic pieces—a survey of results and
  problems.
\newblock {\em Discrete Mathematics}, 339(7):1970--1977, 2016.

\bibitem{gyarfas97nearlycompl}
A.~Gy\'arf\'as, A.~Jagota, and R.~H. Schelp.
\newblock Monochromatic path covers in nearly complete graphs.
\newblock {\em J. Combin. Math. Combin. Comput.}, 25:129--144, 1997.

\bibitem{gyarfas2006improved}
A.~Gy{\'a}rf{\'a}s, M.~Ruszink{\'o}, G.~N. S{\'a}rk{\"o}zy, and
  E.~Szemer{\'e}di.
\newblock An improved bound for the monochromatic cycle partition number.
\newblock {\em Journal of Combinatorial Theory, Series B}, 96(6):855--873,
  2006.

\bibitem{haxell95}
P.~E. Haxell.
\newblock A condition for matchability in hypergraphs.
\newblock {\em Graphs Combin.}, 11(3):245--248, 1995.

\bibitem{janson2011random}
S.~Janson, T.~{\L}uczak, and A.~Ruci\'nski.
\newblock {\em Random graphs}.
\newblock John Wiley \& Sons, 2000.

\bibitem{kohayakawa2016monochromatic}
Y.~Kohayakawa, G.~O. Mota, and M.~Schacht.
\newblock Monochromatic trees in random graphs.
\newblock {\em Mathematical Proceedings of the Cambridge Philosophical
  Society}, to appear.

\bibitem{ks2012phasetransition}
M.~Krivelevich and B.~Sudakov.
\newblock The phase transition in random graphs: A simple proof.
\newblock {\em Random Structures and Algorithms}, 43(2):131--138, 2012.

\bibitem{lang2017local}
R.~Lang and M.~Stein.
\newblock Local colourings and monochromatic partitions in complete bipartite
  graphs.
\newblock {\em European Journal of Combinatorics}, 60:42--54, 2017.

\bibitem{letzter2015monochromatic}
S.~Letzter.
\newblock Monochromatic cycle partitions of $2 $-coloured graphs with minimum
  degree $3 n/4$.
\newblock {\em arXiv preprint arXiv:1502.07736}, 2015.

\bibitem{pokrovskiy2014partitioning}
A.~Pokrovskiy.
\newblock Partitioning edge-coloured complete graphs into monochromatic cycles
  and paths.
\newblock {\em Journal of Combinatorial Theory, Series B}, 106:70--97, 2014.

\bibitem{sarkozy2011monochromatic}
G.~N. S{\'a}rk{\"o}zy.
\newblock Monochromatic cycle partitions of edge-colored graphs.
\newblock {\em Journal of graph theory}, 66(1):57--64, 2011.

\bibitem{sarkozy2013improved}
G.~N. S{\'a}rk{\"o}zy, S.~M. Selkow, and F.~Song.
\newblock An improved bound for vertex partitions by connected monochromatic
  k-regular graphs.
\newblock {\em Journal of Graph Theory}, 73(2):127--145, 2013.

\end{thebibliography}

\end{document}